\RequirePackage{xstring}
\StrBehind*{\jobname}{.}[\DocOpts]

\wlog{DocOpts: \DocOpts}

\documentclass[\DocOpts]{amsart}

\usepackage{amssymb}
\usepackage{enumerate}
\input xy
\xyoption{all}
\usepackage{ifdraft}
\ifdraft{
\usepackage{lineno}
\usepackage{showlabels}
}{

}

\newtheorem*{theorem*}{Theorem}
\newtheorem{theorem}{Theorem}[section]
\newtheorem{lemma}[theorem]{Lemma}
\newtheorem{corollary}[theorem]{Corollary}

\newtheorem{fact}[theorem]{Fact}
\newtheorem{proposition}[theorem]{Proposition}

\newtheorem{claim}[theorem]{Claim}

\theoremstyle{definition}
\newtheorem{example}[theorem]{Example}
\newtheorem{remark}[theorem]{Remark}
\newtheorem{definition}[theorem]{Definition}
\newtheorem*{definition*}{Definition}
\newtheorem{question}[theorem]{Question}
\newtheorem*{question*}{Question}

\newtheorem{innercustomthm}{Theorem}
\newenvironment{customthm}[1]
  {\renewcommand\theinnercustomthm{#1}\innercustomthm}
  {\endinnercustomthm}

\def\endo{\operatorname{End}}

\def\fix{\operatorname{Fix}}

\def\dcf{\operatorname{DCF}}

\def\acfa{\operatorname{ACFA}}
\def\acf{\operatorname{ACF}}

\def\cb{\operatorname{Cb}}

\def\tp{\operatorname{tp}}
\def\qftp{\operatorname{qftp}}
\def\qf{\operatorname{qf}}

\def\acl{\operatorname{acl}}
\def\dcl{\operatorname{dcl}}

\def\alg{{\operatorname{alg}}}
\def\perf{\operatorname{perf}}

\def\aut{\operatorname{Aut}}

\def\loc{\operatorname{loc}}

\def\trdeg{\operatorname{trdeg}}
\def\su{\operatorname{SU}}

\def\nmdeg{\operatorname{nmdeg}}

\def\pgl{\operatorname{PGL}}
\def\GL{\operatorname{GL}}

\def\C{\mathcal{C}}
\def\G{\mathcal{G}}

\def\U{\mathcal{U}}
\def\I{\mathcal{I}}

\def\PP{{\mathbb{P}}}

\def\QQ{\mathbb{Q}}
\def\ZZ{\mathbb{Z}}
\def\CC{\mathbb{C}}

\newcommand{\kk}{k}
\renewcommand{\AA}{\mathbb{A}}

\newcommand{\isom}{\cong}

\providecommand{\id}{\operatorname{id}}

\newcommand{\ra}[1][]{\xrightarrow{#1}}
\newcommand{\Set}[1]{\{#1\}}

\newcommand{\shp}[1]{{#1}^\sharp}
\newcommand{\dom}[1]{{\operatorname{dom} (#1)}}

\newcommand{\pp}[1]{(#1)}

\newcommand{\dto}{\dashrightarrow}

\def\Ind#1#2{#1\setbox0=\hbox{\(#1x\)}\kern\wd0\hbox to 
0pt{\hss\(#1\mid\)\hss}\lower.9\ht0\hbox to 
0pt{\hss\(#1\smile\)\hss}\kern\wd0}
\def\ind{\mathop{\mathpalette\Ind{}}}
\def\notind#1#2{#1\setbox0=\hbox{\(#1x\)}\kern\wd0\hbox to 
0pt{\mathchardef\nn=12854\hss\(#1\nn\)\kern1.4\wd0\hss}\hbox to 
0pt{\hss\(#1\mid\)\hss}\lower.9\ht0 \hbox to 
0pt{\hss\(#1\smile\)\hss}\kern\wd0}
\def\nind{\mathop{\mathpalette\notind{}}}

\newcommand{\qedv}{\vspace{.1in}\hfill\(\maltese\)}
\newenvironment{claimproof}[1]{%
\par\vspace{.1in}\noindent\emph{Proof of Claim:}\space#1}{\qedv}

\date{\today}

\title[A transformal transcendence result]{A transformal transcendence result for algebraic difference equations}
\author{Moshe Kamensky}
\address{Moshe Kamensky\\ Ben-Gurion University of the Negev\\ Department of Mathematics\\ Be'er-Sheva, 8410501\\ Israel}
\email{kamensky.bgu@gmail.com}

\author{Rahim Moosa}
\address{Rahim Moosa\\ University of Waterloo\\ Department of Pure Mathematics\\ 200 University Avenue West\\ Waterloo, Ontario \  N2L 3G1\\ Canada}
\email{rmoosa@uwaterloo.ca}

\thanks{R. Moosa was partially supported by an NSERC Discovery Grant and a University of Waterloo Faculty of Mathematics Research Chair.}

\begin{document}


\begin{abstract}
Given an algebraic difference equation of the form
$$\sigma^n(y)=f\big(y, \sigma(y),\dots,\sigma^{n-1}(y)\big)$$
where $f$ is a rational function over a field~$k$ of characteristic zero on which $\sigma$ acts trivially,
it is shown that if there is a nontrivial algebraic relation amongst any number of $\sigma$-disjoint solutions, along with their $\sigma$-transforms, then there is already such a relation between three solutions.
Here ``$\sigma$-disjoint" means $a\neq\sigma^r(b)$ for any integer~$r$.
A weaker version of the theorem, where ``three" is replaced by $n+4$, is also obtained when $\sigma$ acts nontrivially on~$k$.
Along the way a number of other structural results about primitive rational dynamical systems are established.
These theorems are deduced as applications of a detailed model-theoretic study of finite-rank quantifier-free types in the theory of existentially closed difference fields of characteristic zero.
In particular, it is also shown that the degree of nonminimality of such types over fixed-field parameters is bounded by~$2$.
\end{abstract}

\maketitle

\setcounter{tocdepth}{1}
\tableofcontents


\vfill
\pagebreak
\section{Introduction}

\noindent
The starting point of this paper is the following theorem about algebraic differential equations established several years ago (this is~\cite[Theorem A]{c3}):
{\em Given an algebraic differential equation over constant parameters, if there is any nontrivial algebraic relation amongst any number of distinct nonalgebraic solutions, along with their derivatives, then there is already such a relation between two solutions.}
(If one allows nonconstant parameters than one has to replace ``two" with ``three" in the conclusion, and assume the equation is of order greater than~$1$.)
One is lead to ask: Does the analogous fact hold true of algebraic {\em difference} equations?

Let us make the question more precise.
Fix an algebraically closed field~$k$ of characteristic zero, equipped with an automorphism~$\sigma:k\to k$.
Consider
\begin{equation}
\label{generalde}
P\big(y, \sigma(y),\dots,\sigma^n(y)\big)=0\tag{E}
\end{equation}
an algebraic difference equation where $y$ is a single variable and~$P\in k[y_0,\dots,y_n]$ is an irreducible polynomial.
In the differential case we were interested in arbitrary {\em distinct} solutions, but here that would not make much sense.
For example, in the case when~$\sigma$ is the identity on~$k$, if~$b_1$ is a solution to~(\ref{generalde}) then so is $b_2:=\sigma(b_1)$, regardless of the equation, and there is an obvious nontrivial transformally algebraic relation between these solutions. 
So, at the very least, we should be restricting our attention to solutions that are not transforms of each other to start with.
That is, for each $m\geq 1$, we consider the following transcendence condition that may or may not hold of solutions to~(\ref{generalde}):
\begin{itemize}
\item[$(C_m)$]
Suppose $b_1,\dots,b_m$  are solutions to~(\ref{generalde}), not in~$k$.
If
$$(\sigma^r(b_i): r=0,\dots,n-1, i=1,\dots, m)$$
is not algebraically independent over~$k$ it is because $b_i=\sigma^\ell b_j$ for some $i<j$ and $\ell\in\ZZ$.
\end{itemize}
What $(C_m)$ expresses is that any $m$ pairwise ``$\sigma$-disjoint" solutions to~(\ref{generalde}) are as transformally independent of each other as is consistent with the fact that they each satisfy an order~$n$ algebraic difference equation.
Here we are implicitly considering solutions in arbitrary difference field extensions of $(k,\sigma)$.

Note that $(C_m)\implies(C_\ell)$ for all~$\ell\leq m$.

For example, it follows from work of Chatzidakis and Hrushovski in~\cite{acfa}, that the equation $\sigma(y)-y^2-1=0$ satisfies~$(C_m)$ for all $m\geq 1$.
At the opposite extreme, $\sigma(y)-y=0$ already fails to satisfy~$(C_2)$.

\begin{question*}
Is there an absolute bound~$\ell$ such that
$(C_\ell)\implies(C_m)$ for all~$m\geq\ell$?
If so, what is that bound?
\end{question*}

Restricting our attention to the autonomous case, and considering only equations presented in a certain natural form, we provide the following answer:

\begin{customthm}{1}
\label{c3thm-intro}
Suppose $\sigma$ is the identity on~$k$ and~$(\ref{generalde})$ is of the form 
$$\displaystyle \sigma^n(y)=f\big(y, \sigma(y),\dots,\sigma^{n-1}(y)\big)
$$
for some $f\in k(y_0,\dots,y_{n-1})$.
Then $(C_3)\implies(C_m)$ for all~$m$.
\end{customthm}

The is part of Corollary~\ref{c3thm} below.
The other part drops the assumption that~$\sigma$ is the identity on~$k$ but then only provides a bound depending on~$n$ (namely $n+4$).

While we expect the result to be sharp, we do not yet have an example of an autonomous equation satisfying $(C_2)$ but not $(C_3)$.

The reason we restrict to difference equations of the form appearing in Theorem~\ref{c3thm-intro} is that such equations give rise to {\em rational dynamical systems}, namely pairs $(V,\phi)$ where $V$ is an irreducible variety over~$k$ and $\phi:V\dto V$ is a dominant rational map.
And this article is really about (the applications of model theory to) algebraic dynamics.
Such applications have, by now, a long history, with notable achievements by Chatzidakis-Hrushovski in~\cite{acfa, ch1, ch2} and Medvedev-Scanlon in~\cite{alicetom}.
The immediate antecedent to this paper, however, is~\cite{qfint}, where we developed model-theoretic binding groups for rational dynamical systems.

Underpinning our proof of Theorem~\ref{c3thm-intro} is an analysis of the quantifier-free fragment of  the theory of existentially closed difference fields.
But we delay describing our general model-theoretic results until later in this introduction, beginning instead with two more theorems about rational dynamical systems that we prove along the way.
Implicit when we speak of rational dynamical systems is that $\sigma$ acts trivially on~$k$, and for now this assumption remains in place.

The following list of conditions on a rational dynamical system should give the reader a sense of what aspects of algebraic dynamics we are able to investigate.
Recall that the natural morphisms between rational dynamical systems are {\em dominant equivariant rational maps}, $f:(V,\phi)\dto(W,\psi)$, namely dominant rational maps $f:V\dto W$ over~$k$ such that $\psi f=f\phi$.

\begin{definition*}
\label{ad-def}
Suppose $(V,\phi)$ is a rational dynamical system.
\begin{itemize}
\item[(a)]
$(V,\phi)$ is {\em simple} if for any rational dynamical system $(W,\psi)$, and~$Z$ an irreducible invariant subvariety of $(V\times W,\phi\times\psi)$ that projects dominantly onto both~$V$ and~$W$, either $Z=V\times W$ or  $\dim Z=\dim W$.
\item[(b)]
$(V,\phi)$ is {\em primitive} if  whenever $(V,\phi)\dto(W,\psi)$ is a dominant equivariant rational map then $\dim W$ is either~$0$ or $\dim V$.
\item[(c)]
$(V,\phi)$ is {\em strongly primitive} if $(V',\phi')$ is primitive whenever $(V',\phi')$ admits a dominant equivariant rational map to~$(V,\phi)$ and $\dim V'=\dim V$.
\item[(d)]
$(V,\phi)$ is {\em isotrivial} if for some rational dynamical system $(W,\psi)$, and some $\ell\geq 1$,  there is an equivariant rational embedding of $(V\times W,\phi\times\psi)$  into $(\AA^\ell\times W,\id\times\psi)$ over $(W,\psi)$.
\item[(e)]
By a \emph{rich invariant family of subvarieties of $(V,\phi)$} we mean a rational dynamical system $(W,\psi)$ along with an irreducible invariant subvariety $Z\subseteq(V\times W,\phi\times\psi)$ such that:
$Z$ projects dominantly onto both~$V$ and~$W$, the general fibres of $Z\to W$ are absolutely irreducible subvarieties of~$V$, over distinct general points of~$W$ the fibres are distinct, and  
$\dim Z>\dim V$.
\end{itemize}
\end{definition*}

It is not hard to see that simple rational dynamics are strongly primitive.
In Section~\ref{xexample}, below, we will exhibit an example showing that the converse fails.

If we begin with an arbitrary rational dynamical system then, by taking successive equivariant images and finite covers, we eventually obtain a strongly primitive one.
It is therefore natural to focus first on the structure of strongly primitive rational dynamics.
We prove the following dichotomy:

\begin{customthm}{2}
\label{thm:primitive}
Suppose $(V,\phi)$ is a strongly primitive rational dynamical system.
Then one of the following holds:
\begin{itemize}
\item[(1)]
$(V,\phi)$ is simple and whenever $V_0$ is an irreducible invariant subvariety of of some cartesian power of $(V,\phi)$ that projects dominantly onto $V$ in each co-ordinate, then $(V_0,\phi)$ admits no rich invariant families of subvarieties, or
\item[(2)]
$(V,\phi)$ admits a dominant equivariant rational map $(V,\phi)\dto (V',\phi')$, where $\dim V'=\dim V$, and $(V',\phi')$ is isotrivial.
\end{itemize}
\end{customthm}

Dynamics on simple abelian varieties provide examples of both cases: an endomorphism of a simple abelian variety that is not of finite order will land in case~(1) while translation by a nontorsion $k$-point is an example of case~(2).
The former is due to Hrushovski~\cite{mm}, see the discussion in Example~\ref{example-sav} below, and the latter is worked out in Section~\ref{xexample} below.

In fact, case~(2) always implies that, up to birational equivalence, $\phi$ comes from an algebraic group action on~$V$ --  namely, there is a faithful algebraic group action $G\times V\to V$ such that~$\phi$ agrees with the action of some $k$-point of $G$.
Indeed, that this follows from~(2) was first observed (without using model theory) in~\cite[Corollary~A]{bms}, and then given a model-theoretic account in~\cite[Theorem~5.1]{qfint}.
As dynamical systems coming from algebraic group actions are generally well understood,
and as one can often reduce to the strongly primitive case, we expect the above theorem to allow a reduction of some open questions in algebraic dynamics to case~(1).
A further analysis of case~(1) using model-theoretic tools exists, however it appears to require the more flexible category of algebraic correspondences rather than rational dynamics, and we do not pursue it here.

Combining Theorem~\ref{thm:primitive} with earlier work in~\cite{qfint} we are able to give the following criteria for simplicity:

\begin{customthm}{3}
\label{thm:simple}
Suppose $(V,\phi)$ is a rational dynamical system.
Suppose that any proper irreducible invariant subvariety~$Z$ of $(V^3,\phi)$ that projects dominantly onto both the first co-ordinate and the last $2$ co-ordinates is of dimension $2\dim V$.
Then $(V,\phi)$ is simple.
\end{customthm}

This theorem is saying that in the definition simplicity given above we need only consider $(W,\psi)=(V^2,\phi)$.
We expect that one can replace $(V^2,\phi)$ by $(V,\phi)$ itself, but so far we only see how to do so after replacing $\phi$ by an iterate.

Theorems~\ref{thm:primitive} and~\ref{thm:simple} are special cases of Corollaries~\ref{cor:zdexchange} and~\ref{cor:nmdeg-geom} below, respectively, which hold for what are called {\em rational $\sigma$-varieties}.
These are defined in~Section~\ref{subsect:rsv}, and generalise rational dynamical systems by allowing for a twist by an automorphism of the base field~$k$.

\medskip
\subsection{Our general model-theoretic results}
A distinguishing feature of our work here, and in~\cite{qfint}, is the systematic focus on the quantifier-free fragment of the theory of existentially closed difference fields in characteristic zero ($\acfa_0$).
Theorems~\ref{thm:primitive} and~\ref{thm:simple}, and indeed their generalisations to rational $\sigma$-varieties, 
are geometric formulations of statements we prove about quanitifer-free types in~$\acfa_0$.
We describe those now, referring the reader to Section~\ref{prelims} for a review of the model-theoretic terminology.

We work in a sufficiently saturated model $(\U,\sigma)\models\acfa_0$, with fixed field~$\C$, and over an algebraically closed inversive difference subfield $(k,\sigma)$.
So we are no longer assuming that $\sigma$ acts trivially on~$k$, though, as we will see, that assumption will at times improve our results significantly.

Among the complete quantifier-free types over~$k$, we restrict our attention the the {\em rational} types, those that imply a formula of the form $\sigma(x)=\phi(x)$ for some rational (multivariable) function~$\phi$.
Using the canonical base property (CBP), established in~\cite{pillay-ziegler}, and the Zilber dichotomy for minimal types, which follows from the CBP but was first established in~\cite{acfa}, we prove:

\begin{customthm}{4}
\label{thm:primitive-mt}
Suppose~$p$ is a rational type over~$k$ satisfying the following property:
if $a\models p$ and $b\in\acl(ka)\setminus k$ is such that $\qftp(b/k)$ is rational then $a\in\acl(kb)$.
Then one of the following hold:
\begin{enumerate}
\item
$p$ is minimal and one-based, or
\item
there is a finite-to-one 
rational map $p\to q$ where~$q$ is rational and qf-internal to~$\C$.
\end{enumerate}
\end{customthm}

This appears as Theorem~\ref{thm:nonorth} below.
It is the analogue of a fact about complete finite $U$-rank types in stable theories established in~\cite[Proposition~2.3]{moosa-pillay2014}.

Theorem~\ref{thm:primitive} is a geometric formulation of Theorem~\ref{thm:primitive-mt} in the case that $k\subset\C$.

Combining Theorem~\ref{thm:primitive-mt} with our earlier results on quantifier-free binding groups from~\cite{qfint}, allows us to obtain bounds on the {\em degree of nonminimality} for a rational type.
This quantity, when~$p$ is a nonalgebraic and nonminimal rational type, denoted by $\nmdeg(p)$, is the least positive integer~$d$ such that there are realisations $a_1,\dots,a_d$ of~$p$, and a nonalgebraic forking extension of~$p$ to a complete quantifier-free type over $k(a_1,\dots,a_d)$.
Degree of nonminimality for complete finite $U$-rank types in stable theories were introduced in~\cite{nmdeg} where initial bounds for the theory $\dcf_0$ were obtained.
It was later shown, in~\cite{nmdeg2}, that $\dcf_0$ admits an absolute bound of~$2$ for the degree of nonminimality, which goes down to~$1$ if the type is over constant parameters.
Such bounds are certainly specific to~$\dcf_0$ and similar theories: recent work of Baldwin-Freitag-Mutchnik~\cite{bfm} shows that the degree of nonmininality may take on any positive integer value if one allows arbitrary $\omega$-stable theories.
We obtain the following weaker (and we expect not optimal) bounds for rational types in $\acfa_0$.

\begin{customthm}{5}
\label{thm:simple-mt}
Suppose~$p$ is a rational type over~$k$.
\begin{itemize}
\item[(a)]
$\nmdeg(p)\leq\dim(p)+3$.
\item[(b)]
If $k\subset\C$ then $\nmdeg(p)\leq 2$.
\end{itemize}
\end{customthm}

This appears as Theorem~\ref{thm:nmdeg} below.

Theorem~\ref{thm:simple} is a geometric formulation of Theorem~\ref{thm:simple-mt}(b).

Finally, let us give the model-theoretic formulation (and generalisation) of Theorem~\ref{c3thm-intro} itself.
Informed by ideas in~\cite{c3} we introduce here the notion of {\em $m$-disintegration} for a rational type: every set of $m$ realisations is independent unless one is a transform of another by some power of~$\sigma$.
We are able to show:

\begin{customthm}{6}\label{d3-intro}
Suppose $p$ is a rational type over~$k$.
\begin{itemize}
\item[(a)]
$(\dim(p)+4)$-disintegration implies $m$-disintegration for all~$m$.
\item[(b)]
If $k\subset\C$ then $3$-disintegration implies $m$-disintegration for all~$m$.
\end{itemize}
\end{customthm}

This appears as Theorem~\ref{d3} below (taking into account Theorem~\ref{thm:nmdeg}).

\medskip
\subsection{Plan of the paper}
After recalling (and elaborating on) preliminaries in Section~\ref{prelims}, we begin our study of rational types that are minimal in Section~\ref{sect:qfmin}.
In particular, we give the geometric characterisation of minimality there in terms of simple rational $\sigma$-varities.
In Section~\ref{sec:zd} we give an account of the Zilber dichotomy for minimal rational types (Theorem~\ref{thm:zd}).
This can be extracted from the results in~\cite{acfa} or~\cite{pillay-ziegler}.
While we do rely on the latter,  we obtain cleaner statements from specialising to the case of rational types, and we are able give an explicit geometric formulation of the dichotomy (Corollary~\ref{cor:zd}).
In Section~\ref{sec:x} we introduce the exchange property and prove Theorem~\ref{thm:primitive-mt}, and its geometric formulation Theorem~\ref{thm:primitive}.
Then, in Section~\ref{sec:nmdeg}, we introduce degree of nonminimality and prove Theorem~\ref{thm:simple-mt} and its geometric formulation Theorem~\ref{thm:simple}.
These things are put together in Section~\ref{sec:disintegration} to deduce Theorem~\ref{d3-intro} and finally Theorem~\ref{c3thm-intro}.
Two examples are worked out in a final Section~\ref{sec:example}: a nonminimal rational type that satisfies exchange and a rational type whose degree of nonminimality is~$2$ and which is $2$-disintegrated but not $3$-disintegrated.

\bigskip
\section{Preliminaries}
\label{prelims}

\noindent
We work in a sufficiently saturated model $(\U,\sigma)\models \acfa$ with fixed field denoted by~$\C$.
We will eventually assume that the characteristic is zero, but for now we allow positive characteristic.

We also fix a base algebraically closed inversive difference subfield $(k,\sigma)$.

In what follows we review and expand on some notions/terminology from~\cite{qfint}.

\medskip
\subsection{Generalised algebraic dynamics}
\label{subsect:rsv}
By a {\em rational $\sigma$-variety} we mean an irreducible variety~$V$ over~$k$ equipped with a dominant rational map $\phi:V\dto V^\sigma$, from~$V$ to its transform with respect to the action of $\sigma$ on~$k$.
So, when~$\sigma$ is the identity on~$k$, what we call the {\em autonomous} case, a rational $\sigma$-variety is a rational dynamical system.

By an {\em invariant subvariety}~$Z$ of a rational $\sigma$-variety $(V,\phi)$ we will mean an irreducible (closed) subvariety $Z\subseteq V$ over~$k$ such that $Z\cap\dom{\phi}$ is nonempty and $\phi(Z)$ is Zariski dense in~$Z^\sigma$.
So we get an induced rational $\sigma$-variety $(Z,\phi|_Z)$.
Our invariant subvarieties will always be irreducible, but we will reinforce that by using the term ``irreducible invariant subvariety" throughout.

An {\em equivariant} rational map, $f:(V,\phi)\dto(W,\psi)$, is a rational map $f:V\dto W$ such that $\psi f=f^\sigma\phi$ as rational maps on~$V$.

We associate to a rational $\sigma$-variety its ``sharp" points, namely the $k$-definable set $(V,\phi)^\sharp:=\{a\in \dom{\phi}:\sigma(a)=\phi(a)\}$.
Here the action of~$\sigma$ is co-ordinatewise.
The {\em generic quantifier-free type} of $(V,\phi)$ is the collection of (quantifier-free) formulas $p(x)$ expressing that~$x$ is a Zariski generic point of~$V$ over~$k$ and that $x\in(V,\phi)^\sharp$.
Existential closedness of~$(\U,\sigma)$ ensures that this is a consistent and complete quantifier-free type.
Realisations are called {\em generic points} of $(V,\phi)$.

\medskip
\subsection{Rational types}\label{subsect:rt}
By a \emph{rational type over~$k$} in $x=(x_1,\dots,x_n)$ we mean a complete quantifier-free type, 
$p(x)\in S_{\qf}(k)$, which implies a formula of the form $\sigma(x)=\phi(x)$ 
for some $\phi=(f_1,\dots,f_n)$ where each $f_i\in  k(x)$ is a rational function.
The generic quantifier-free type of a rational $\sigma$-variety over~$k$ is rational, and all rational types arise this way.

Note that $\qftp(a/k)$ is rational if and only if the difference field that~$a$ generates over~$k$ is just the field generated by~$a$ over~$k$.
That is, rationality is equivalent to $k\langle a\rangle=k(a)$.
It is also worth observing that a consequence of $\qftp(a/k)$ being rational (and $k$ being inversive) is that $\acl(ka)=k(a)^{\alg}$.

If $p,q\in S_{\qf}(k)$ are rational types then by a \emph{rational map} $f:p\to q$ we mean that~$f$ is a rational function over~$k$ such that $f(a)\models q$ for all (equivalently some) $a\models p$.
These arise as the restrictions of dominant rational equivariant maps which we tend to also denote by $f:(V,\phi)\dto(W,\psi)$, where $p$ is the generic quantifier-free type of $(V,\phi)$ and $q$ is the generic quantifier-free type of $(W,\psi)$.
Note that, unlike in the case of complete types, $f$ need not
be surjective as a set map from $p(\U)\to q(\U)$.
For example, $x\mapsto x^2$ is a dominant equivariant map $(\AA^1,\id)\to(\AA^1,
\id)$ that does not restrict to a surjective map on the generic quantifier-free types (as there are transcendental fixed points whose square roots are not fixed).

We say that $f:p\to q$ is \emph{birational} if $k(a)=k(f(a))$ for some (equivalently all) $a\models p$.
This is equivalent to $f:V\dto W$ being birational.
We say that $f:p\to q$ is \emph{finite-to-one} if $a\in {k(f(a))}^{\alg}$ for 
some (equivalently all) $a\models p$.
This is equivalent to $\dim V=\dim W$.

Nonforking independence in $\acfa$ has an explicit (quantifier-free) characterisation: $A\ind_CB$ if and only if the algebraically closed inversive difference fields generated by $A\cup C$ and $B\cup C$ are algebraically disjoint over the algebraically closed inversive difference field generated by~$C$.
It follows that if $p=\qftp(a/k)$ is rational, and $L\supseteq k$ is a difference field extension, then $a\ind_kL$ if and only if $\trdeg(k(a)/a)=\trdeg(L(a)/L)$.
For any positive integer~$d$ we denote by $p^{(d)}$ the (unique) quantifier-free type of a $d$-tuple of independent realisations of~$p$.
If $p$ is the generic quanitifier-free type of $(V,\phi)$ then $p^{(d)}$ is the generic quantifier-free type of the rational $\sigma$-variety $(V^d,\phi)$ where $\phi$ acts diagonally on $V^d$.

By the {\em dimension} of a rational type~$p$ we mean $\trdeg(k(a)/a)$ for any $a\models p$.
This agrees with $\dim V$ where~$p$ is the generic quantifier-free type of $(V,\phi)$.

\medskip
\subsection{Orthogonalilty, qf-internality, and binding groups}
\label{subsect:nonorth}
By a slight abuse of notation, given a tuple~$a$, we write $a\ind_k\C$ to mean that $a$ is independent of every finite tuple from~$\C$ over~$k$.
(This is only an abuse of notation because normally we only talk about independence between small sets, whereas~$\C$ has the cardinality of the universe.)
That this is a condition on $\qftp(a/k)$, at least when that quantifier-free type is rational, follows from~\cite[Proposition~2.7]{qfint}, which says that $a\ind_k\C$ if and only if $k(a)\cap\C\subseteq k^{\alg}$.
Hence $a\ind_k\C$ for some realisation of a rational type $p\in S_{\qf}(k)$ if and only if this is the case for every realisation.
We then say that $p$ is  \emph{weakly orthogonal} to~$\C$, denoted by $p\perp^w\C$.
A rational and stationary $p$ is defined to be \emph{orthogonal} to~$\C$, denoted by $p\perp\C$, if for every difference field extension $K\supseteq k$ the nonforking extension of~$p$ to $S_{\qf}(K)$ is weakly orthogonal to~$\C$.

The opposite notion, that of internality to~$\C$, is not a property of the quantifier-free type.
As observed in~\cite[$\S$2.4]{qfint}, following ideas from~\cite[$\S$5]{acfa}, this can be resolved by restricting to rational types and replacing $\dcl(Kc)$ with $K(c)$.

\begin{definition}
\label{def:qfint}
Suppose $p\in S_{\qf}(k)$ is a rational type over an algebraically closed difference field.
We say that $p$ is {\em qf-internal to~$\C$} if for some (equivalently all) $a\models p$ there is a difference field extension $K\supseteq k$ such that $a\ind_kK$ and $a\in K(c)$ for some tuple $c$ from $\C$.
\end{definition}

In~\cite[$\S$2.4]{qfint} a detailed study of qf-internality is carried out.
The reader can look there for a proof that ```for some" is equivalent to ``for all" in the definition.
It is also shown there that $p\not\perp\C$ if and only if $p$ admits a rational map to a positive-dimensional rational type that is qf-internal to~$\C$.

It is shown in~\cite[Proposition~3.6]{qfint} that, in characteristic zero, if $p$ is the generic quantifier-free type of $(V,\phi)$ then qf-internality of~$p$ to~$\C$ is equivalent to the {\em isotriviality} of~$(V,\phi)$.
We defined isotriviality in the Introduction, but only for the autonomous case.
In fact, that definition works verbatim with rational $\sigma$-varieties in place of rational dynamical systems.

Fix now a rational type $p\in S_{\qf}(k)$ that is qf-internal to~$\C$, and assume the characteristic is zero.
The {\em quantifier-free binding group of~$p$ relative to~$\C$}, denoted by $\aut_{\qf}(p/\C)$, was defined in~\cite{qfint} as the group of permutations~$\alpha$ of $p(\U)$ such that
\begin{itemize}
	\item[(A)] for any tuple \(u\) of realisations of~\(p\), and any tuple \(c\) of elements of \(\C\),
		\(\qftp(u,c/k)=\qftp(\alpha(u),c/k)\).
\end{itemize}
The main accomplishment of~\cite{qfint} was to show that $\aut_{\qf}(p/\C)$, along with its action on $p(\U)$, has a quantifier-free definable avatar.
But beyond this, the group action was given a geometric meaning that we now explain.
Let $\pp{V,\phi}$ be the isotrivial rational $\sigma$-variety corresponding to~$p$.
The main results of~\cite{qfint} are the following: there exists an algebraic group~$G$ over~$k$ acting faithfully on~$V$ by birational transformations, $\theta:G\times V\dto V$,
such that $\{\theta_g:g\in G\}$ is precisely the group of those birational transformations of~$V$ that:
\begin{itemize}
\item[(B)]
preserve every irreducible invariant subvariety of $(V^r,\phi)\times(W,\id)$ over~$k$ projecting dominantly onto each $V$-co\"ordinate, for all $r\geq 1$ and all trivial dynamics $(W,\id)$ over~$k$.
(Here we view a birational transformation of~$V$ as acting co-ordinate wise on $V^r$ and trivially on $W$.)
\end{itemize}
Moreover, there is
an algebraic group isomorphism 
\(\rho:G\ra{}G^\sigma\) making~$\theta$ equivariant as a rational map from $(G\times V,\rho\times\phi)$ to $(V,\phi)$, and such that 
\(\shp{\pp{G,\rho}}(\U)\) is identified with $\aut_{\qf}(p/\C)$ by restricting the action to $p(\U)$.
We call $(G,\rho)$ the {\em binding group of $(V,\phi)$}.
Note that, as rational maps from \(V\) to \(V^\sigma\), we have 
	\(\rho(g)\circ\phi=\phi\circ{}g\) for all \(g\in{}G\).
Indeed, this is another way of writing the equivariance of \(\theta\).
In particular, if~\(\phi\) is invertible then \(\rho\) is conjugation by~$\phi$.

\medskip
\subsection{Canonical bases}
Recall from~\cite[Lemma 2.5(b)]{qfint} that the canonical base of a rational type, $\qftp(a/K)$ over a perfect difference field~$K$, is the difference field generated by the minimal field of definition of the Zariski locus $\loc(a,\sigma(a)/K)$.
We record the following useful additional facts.

\begin{lemma}\label{lem:cbred}
Suppose $p=\qftp(a/k)$ is rational and $K\supseteq k$ is a perfect difference field extension.
\begin{itemize}
\item[(a)]
$k\cdot\cb(a/K)=k(b)=k\langle b\rangle$ for some finite tuple $b$ from~$K$.
Moreover, $b$ can be chosen from the perfect closure of the field generated by a finite number of realisations of~$p$ over $k$.
\item[(b)]
Let $F$ be the minimal field of definition of $\loc(a/K)$.
Then
$$(k\cdot\cb(a/K))^{\perf}=\langle k\cdot F\rangle^{\perf}.$$
\end{itemize}
\end{lemma}

\begin{proof}
(a).
Let $E:=k\cdot\cb(a/K)$, the field generated by $\cb(a/K)$ over $k$.
So $E$ is a finitely generated difference field extension of~$k$ contained in~$K$.
On the other hand, by~\cite[Lemma~2.5(c)]{qfint}, there are $a_1,\dots,a_\ell$ realising $\qftp(a/K)$, and hence realising~$p$,  such that $E\subseteq k(a_1,\dots,a_\ell)^{\perf}$.
Using the fact that the preimage of $k(a_1,\dots,a_\ell)$ under any power of Frobenius is a difference field, and that $E$ is finitely generated over $k$ as a difference field, it follows that $E$ is contained in a finite extension of $k(a_1,\dots,a_\ell)$, and hence is also finitely generated as a field extension of~$k$.
That is, $E=k(b)=k\langle b\rangle$ for some finite tuple $b$.

(b).
This is similar to~\cite[Lemma~4.5]{pillay-ziegler}, but we give some details.

This time, let $E$ be the minimal field of definition of $\loc(a,\sigma(a)/K)$, so that $\cb(a/K)$ is the difference field generated by~$E$.
Then $\loc(a/K)$ is
defined over $E$ and hence
$$F\subseteq E^{\perf}\subseteq\cb(a/K)^{\perf},$$
showing the right-to-left inclusion.

For the left-to-right inclusion, let $\phi:\loc(a/k)\to\loc(a/k)^\sigma$ be a rational function over~$k$ such that $\sigma(a)=\phi(a)$, and denote its graph by $\Gamma(\phi)$.
Note that $\phi$ restricts to a rational function from $\loc(a/K)$ to $\loc(a/K)^\sigma$ whose graph we denote by $\Gamma(\phi|_{\loc(a/K)})$.
 Now fix $\alpha$ an arbitrary field-automorphism of $\U$ that fixes $k\cdot F$ pointwise.
Then
 \begin{eqnarray*}
 \loc(a,\sigma(a)/K)^\alpha
 &=&
 \loc(a,\phi(a)/K)^\alpha\\
 &=&
 \Gamma(\phi|_{\loc(a/K)})^\alpha\\
 &=&
 \Gamma(\phi|_{\loc(a/K)})\ \ \text{ as $\alpha$ is the identity on $k$ and $F$}\\
 &=&
 \loc(a,\phi(a)/K)\\
 &=&
  \loc(a,\sigma(a)/K).
 \end{eqnarray*}
 Hence $\alpha$ is the identity on $E$.
 We have shown that $E\subseteq (k\cdot F)^{\perf}$, from which it follows that $\cb(a/K)\subseteq\langle k\cdot F\rangle^{\perf}$, as desired.
\end{proof}

The following important property of canonical bases in $\acfa_0$ was proven in~\cite{pillay-ziegler}, though we phrase it somewhat differently and specialise to rational types.

\begin{fact}[Canonical base property for rational types in~$\acfa_0$]
\label{fact:cbp}
Suppose~$k$ is an algebraically closed inversive difference field of characteristic zero, $p=\qftp(a/k)$ is a rational type, and $K\supseteq k$ is an algebraically closed difference field extension.
Let $b$ be as in Lemma~\ref{lem:cbred}(a), that is, such that
$k\cdot\cb(a/K)=k(b)=k\langle b\rangle$.
Then $\qftp(b/k(a))$ is qf-internal to~$\C$.
\end{fact}

\begin{proof}
We spell out how this follows from~\cite[Theorem~1.2]{pillay-ziegler}.
What is actually shown there (see the proof on page~591) is the following: if $F$ is the minimal field of definition of $\loc(a/K)$ then $F\subseteq L(c)$ where $c$ is a tuple from $\C$ and $L\supseteq k(a)$ is a difference field extension such that $L\ind_{k(a)}F$.
Now, by Lemma~\ref{lem:cbred}, $k(b)=\langle k\cdot F\rangle$, so that $b\in L(c)$ and $L\ind_{k(a)}b$.
That is, $L$ and $c$ witness the qf-internality of $\qftp(b/k)$ to~$\C$.
(Note that $\qftp(b/k)$ is rational).
\end{proof}

\begin{remark}
In~\cite{pillay-ziegler} the conclusion of the CBP is just ``almost internality".
We get outright qf-internality here because we start with a rational type.
\end{remark}

\bigskip
\section{Minimality}
\label{sect:qfmin}

\noindent
We continue to work over an algebraically closed inversive difference field $(k,\sigma)$.
It is somewhat surprising (to us) that the following fundamental question (asked already in a stronger form in~\cite[Remark~1.15]{acfa2}) remains open:

\begin{question}
\label{question:qfmin=min}
Is $SU$-rank, even $SU$-rank being~$1$, a property of the quantifier-free type?
That is, given a pair of tuples $a,a'$ such that $\qftp(a/k)=\qftp(a'/k)$, is it the case that $\su(a/k)=1$ if and only if $\su(a'/k)=1$?
\end{question}

What makes an affirmative answer to this question plausible is the following algebraic-dynamics characterisation of having $SU$-rank greater than~$1$.

\begin{proposition}
\label{prop:nonmin-geo}
Suppose~\(p\) is is the generic quantifier-free type of a rational $\sigma$-variety $(V,\phi)$ over~$k$.
Then the following are equivalent:
\begin{itemize}
\item[(i)]
For some $a\models p$, \(\su(a/k)>1\).
\item[(ii)]
There exists a rational $\sigma$-variety $(W,\psi)$ over~$k$, and a proper irreducible invariant subvariety~$Z$ of $(V\times W,\phi\times\psi)$ that projects dominantly onto both~$V$ and~$W$, with $\dim Z>\dim W$.
\end{itemize}
\end{proposition}

\begin{proof}
(i)$\implies$(ii).
Suppose~$\su(a/k)>1$ is witnessed by $K\supseteq k$ such that $a\notin\acl(K)$ but $a\nind_kK$.
We may assume that $K=\acl(K)$.
Let $e=\cb(a/K)\in K$.
Since $q:=\qftp(a/K)$ is rational, there is, by~\cite[Lemma~2.5(c)]{qfint}, 
some $d\geq 0$  such that $e\in k\langle b\rangle$ for any $b=(a_1,\dots,a_d)$ 
independent realisations of~$q$.
In particular $e\in k(b)$.
We may choose $b\ind_Ka$.
In particular, $a\notin\acl(kb)$.
On the other hand, as $a\ind_{e}K$ and $a\nind_kK$, we must have that 
$a\nind_ke$ and hence $a\nind_kb$.
That is, $k(b)$ is also a witness to $\su(a/k)>1$.
Now, let $W=\loc(b/k)\subseteq V^d$.
Since $b$ is a tuple of realisations of~$p$, it is a $\sharp$-point for the 
co-ordinatewise action of $\phi$ on $V^d$.
Hence $W$ is invariant for $\phi$, and if we let~$\psi$ be the restriction of $\phi$ to $W$ then $(W,\psi)$ is a rational $\sigma$-variety over $k$.
Let $Z:=\loc((a,b)/k)$, it is an irreducible invariant subvariety of $(V\times W,\phi\times\psi)$ that projects dominantly on $V$ and $W$.
That $a\nind_kb$ forces $Z$ to be a proper subvariety of $V\times W$.
On the other hand, that $a\notin\acl(kb)$ forces 
\(\dim{}Z>\dim{}W\).

(ii)$\implies$(i)
Suppose $(W,\psi)$ and $Z\subseteq V\times W$ are as given by~(ii).
Let $(a,b)\in Z$ be generic for the induced rational $\sigma$-variety structure on~$Z$ over~$k$.
(Such a generic point exists because we are working over an algebraically closed difference field, and so $Z$ is absolutely irreducible.)
By dominance of $Z\to V$, $a$ is Zariski generic in~$V$ over~$k$, and by dominance of $Z\to W$, $b$ is Zariski generic in~$W$ over~$k$.
Hence, as $Z\neq V\times W$, we have that $a\nind_kb$.
(Note that $V\times W$ is also irreducible as we are working over an algebraically 
closed field.)
On the other hand, $\dim Z>\dim W$ implies $a\notin{k(b)}^{\alg}$.
But by equivariance of $Z\to W$ we have that $b\in{(W,\psi)}^\sharp$, which 
implies that \({k(b)}^{\alg}={k\langle{}b\rangle}^{\alg}=\acl(kb)\).
This witnesses that $\su(a/k)>1$.
By equivariance of $Z\to V$ we have that $a\models p$, establishing~(i).
\end{proof}

In the wake of this, one might hope to try to answer Question~\ref{question:qfmin=min} affirmatively for rational types~$p$, by showing that if $Z\to W$ witnesses that some $\su(a/k)>1$, and $a'\models p$ is another realisation, then another fibre of the same family will witness $\su(a'/k)>1$.
The following example, already in the fixed field, says that approach cannot work:

\begin{example}
Working in characteristic~$0$, let $(V,\phi)=(\AA^2,\id)$, $(W,\psi)=(\AA^1,\id)$, and $Z\subseteq V\times W$ be defined by $w^2+wv_1+v_2=0$, where $(v_1,v_2)$ are co-ordinates for $V$ and~$w$ a co-ordinate for~$W$.
Note that every irreducible variety of $V\times W$ is invariant as the dynamics here are trivial.
It is clear that $Z$ projects dominantly onto both~$V$ and~$W$, and that $\dim Z>\dim W$.
So, by Proposition~\ref{prop:nonmin-geo}, $Z$ witnesses that the generic quantifier-free type $p$ of $(V,\phi)$ has a completion of $\su$-rank greater than~$1$.
More precisely, if we choose $(a,b)\in (Z,\id)^\sharp$ generic in~$Z$ then the fibre $Z_b\subsetneq V$ gives rise to a nonalgebraic forking extension of $\tp(a/k)$, namely $\tp(a/kb)$.
But suppose, instead, that we choose $a=(a_1,a_2)\models p$ such that
\begin{equation*}
\sigma\left(\sqrt{a_1^2-4a_2}\right)=-\sqrt{a_1^2-4a_2}.
\end{equation*}
Then $a$ lives on two fibres of~$Z$, namely $Z_b$ and $Z_{b'}$ where $b, b'$ are the two solutions to $w^2+wa_1+a_2=0$.
The above identity forces $\sigma(b,b')=(b',b)$.
In particular, $b,b'\notin(W,\psi)^\sharp$.
But what matters for us is that $k\langle b\rangle=k(b,b')$ and $k\langle b'\rangle=k(b,b')$.
Since $a\in k(b,b')^{\alg}$, we get that $a\in\acl(kb)$ and $a\in\acl(kb')$; in other words $\tp(a/kb)$ and $\tp(a/kb')$ are algebraic extensions of $\tp(a/k)$ and hence neither witnesses that $\su(a/k)>1$.
What we have shown is that none of the fibres of $Z\to W$ on which $a$ lives can witness that $\su(a/k)>1$.
One would have to (and can in this case) choose~$Z$ differently to find such a witness.
\end{example}

In any case, we avoid answering Question~\ref{question:qfmin=min} by restricting attention to those types all of whose realisations have $SU$-rank~$1$.

\begin{definition}
A complete quantifier-free type $p\in S_{\qf}(k)$ is \emph{minimal} if $\su(p)=1$ as a partial type. That is, if $\su(a/k)=1$ for all $a\models p$.
\end{definition}

One easily  obtains the expected characterisations:

\begin{proposition}
\label{prop:min-equiv}
Suppose~\(p\) is a nonalgebraic complete quantifier-free type over~$k$.
Then the following are equivalent:
\begin{itemize}
\item[(i)]
$p$ is minimal.
\item[(ii)]
Every forking extension of~$p$ is algebraic.
\item[(iii)]
Every relatively quantifier-free definable subset of~$p(\U)$ is finite or cofinite.
 \end{itemize}
\end{proposition}

\begin{proof}
(i)$\implies$(ii).
Suppose $a\models p$ and $B\supseteq k$ is such that $a\nind_kK$.
Then $\su(a/B)<\su(a/k)=1$, so that $a\in\acl(B)$. 

(ii)$\implies$(iii).
Suppose $D$ is a quantifier-free definable set over some difference field 
extension $K\supseteq k$.
If $p(\U)\cap D$ and $p(\U)\setminus D$ are both infinite, then we can find, by compactness, nonalgebraic extensions $q_1,q_2\in S_{qf}(K)$ of $p$ with $q_1$ containing ``$x\in D$" and $q_2$ containing ``$x\notin D$".
As~$p$ is stationary, these cannot both be nonforking extensions.
Hence~$p$ has a nonalgebraic forking extension to $S_{\qf}(K)$.

(iii)$\implies$(i).
If $p$ is not minimal witnessed by $K\supseteq k$ and $a\models p$ with $a\nind_kK$ and $\su(a/K)>0$.
So $q_1:=\qftp(a/K)$ is a nonalgebraic extension of~$p$.
Let $q_2\in S_{\qf}(K)$ be the nonforking extension of~$p$ -- it is also nonalgebraic.
As $q_1\neq q_2$, since one is a forking extension and one not, there is a quantifier-free formula in $q_1$ whose negation is in $q_2$.
The set $D$ defined by this formula will have the property that $D\cap p(\U)\supseteq q_1(\U)$ and $p(\U)\setminus D\supseteq q_2(\U)$ are both inifnite.
\end{proof}

For rational types, Proposition~\ref{prop:nonmin-geo} gives a geometric characterisation of minimality.
Borrowing from terminology in bimeromorphic goemetry, we define:

\begin{definition}
\label{def:simple}
A rational $\sigma$-variety $(V,\phi)$ over~$k$ is {\em simple} if for any rational $\sigma$-variety  $(W,\psi)$ over~$k$, and~$Z$ an irreducible invariant subvariety of $(V\times W,\phi\times\psi)$ that projects dominantly onto both~$V$ and~$W$, either $Z=V\times W$ or  $\dim Z=\dim W$.
\end{definition}

Proposition~\ref{prop:nonmin-geo} can be restated as:

\begin{proposition}
\label{prop:min-simple}
Suppose $(V,\phi)$ is a positive-dimensional  rational $\sigma$-variety over~$k$ with generic quantifier-free type~$p$.
Then $p$ is minimal if and only if $(V,\phi)$ is simple.
\end{proposition}

Simplicity says that $(V,\phi)$ is not covered by an invariant  family of proper infinite subvarieties.
In particular, rational $\sigma$-curves are always simple.
Here is a higher dimensional example:

\begin{example}
\label{example-sav}
Let $A$ be a simple abelian variety over an algebraically closed field~$k\subseteq\C$, with $\phi\in\endo(A)$ not of finite order.
Then $(A,\phi)$ is simple.
\end{example}

\begin{proof}
This follows from the work in~\cite{mm}, but we show how.

Let $B:=(A,\phi)^\sharp(\U)$.
This is a quantifier-free definable subgroup of $A(\U)$.

We first show that $B$ is \emph{c-minimal} in the sense of~\cite{mm}; that is, it does not admit any infinite definable subgroups of infinite index.
Let $R$ be the (noncommutative) ring obtained from the ring of definable endomorphisms of $A(\U)$ after tensoring with $\QQ$ over $\ZZ$.
By~\cite[Proposition 4.1.1(5)]{mm}, c-minimality will follow if we can show that $\sigma-\phi$ is left-irreducible in~$R$.
But~\cite[Proposition 4.1.1(3)]{mm}, specialised to our autonomous case, tells us that $R=\endo_{\QQ}(A)[\sigma,\sigma^{-1}]$, namely that it is generated as a ring by $\{\sigma,\sigma^{-1}\}$ over the division ring $\endo_{\QQ}(A)$.
It now follows easily that $\sigma-\phi$, being monic of degree~$1$ in $\sigma$, is left-irreducible in $R$.

Next, we argue that $B$ is \emph{LMS}, again in the sense of~\cite{mm}; namely that every definable subset of every cartesian power is a finite boolean combination of cosets of definable subgroups.
By~\cite[Proposition 4.1.2(c)]{mm}, LMS will follow once we know that $B$ is not contained in the kernel of $\sigma^N-1$, for any $N>0$.
This is indeed the case from our assumption that $\phi$ is not a root of unity;
$B$ will contain a Zariski generic point of~$A$, and if $\phi^N$ fixed such a point it would be identically trivial.

Now we can apply~\cite[Proposition 2.1.2(a)]{mm} which says that if $B$ is c-minimal and LMS then it is of $SU$-rank~$1$.
It follows that the generic quantifier-free type of $(A,\phi)$, say~$p$, as it extends~$B$, is minimal.
Simplicity of $(A,\phi)$ now follows from Proposition~\ref{prop:min-simple}.
\end{proof}

Let us point out that in the nonorthogonal to~$\C$ case, minimality is nothing more than one-dimensionality.

\begin{lemma}
\label{lem:aim-dim1}
Suppose~\(p\) is a nonalgebraic rational stationary type that is nonorthogonal to~$\C$.
Then $p$ is minimal if and only if $\dim(p)=1$.
\end{lemma}

\begin{proof}
That one-dimensional types are minimal follows immediately from the characterisation of minimality given by~\ref{prop:min-equiv}(ii).
Suppose, now, that $p\in S_{\qf}(k)$ is minimal and nonorthogonal to~$\C$.
We therefore have $a\models p$, $K\supseteq k$ with $a\ind_kK$, and a tuple~$(c_1,\dots,c_m)$ from~$\C$ such that $a\nind_K(c_1,\dots,c_m)$.
By minimality, it follows that $a\in\acl(Kc_1\dots c_m)$.
Let~$m$ be least such.
So $a\notin\acl(Kc_1\dots c_{m-1})$.
If we let $L:=K(c_1,\dots, c_{m-1})$, then minimality gives $a\ind _kL$.
Hence
\begin{eqnarray*}
\dim(p)
&=&
\dim(a/L)\ \ \text{ as } a\ind _kL\\
&=&
\trdeg(L(a)/L)\ \ \text{ as $\qftp(a/L)$ is rational}\\
&\leq&
\trdeg(L(c_m)/L)\ \ \text{ as } a\in\acl(Lc_m)\\
&\leq&
1
\end{eqnarray*}
By nonalgebraicity, it follows that $\dim(p)=1$.
\end{proof}

\bigskip
\section{One-basedness and the Zilber dichotomy}
\label{sec:zd}

\noindent
We continue to work over an algebraically closed inversive difference field $(k,\sigma)$ which we now, and for the rest of the paper, assume to be of characteristic zero.

Chatzidakis and Hrushovski~\cite{acfa} introduce a notion of modularity for sets in $\acfa$ that we here rephrase in terms of the familiar stability-theoretic notion of one-basedness, and we  specialise to rational types:

\begin{definition}
Suppose $p\in S_{\qf}(k)$ is a rational type.
We will say that~$p$ is \emph{one-based} if for any finite tuple of 
realisations~$a$ of~$p$, and any difference field extension $K\supseteq k$, 
$\cb(a/K)\in \acl(ka)$.
\end{definition}

\begin{remark}\label{rem:modular}
  This property coincides with ``modularity of $p(\U)$ over~$k$", in the sense 
  of~\cite[\S3.3]{acfa}.
  This is the condition that asks that~\(a\) be independent from~\(b\) over \(\acl(ka)\cap\acl(kb)\), for any pair of finite tuples of realisations~$a$ and~$b$ of~$p$.
    Indeed, if \(p\) is one-based, and \(a, b\) 
  are tuples of realisations of \(p\), then by assumption \(\cb(a/kb)\) is 
  contained in \(\acl(ka)\).
  Of course, it is also contained in 
  \(\acl(kb)\), from which it follows that \(a\) is independent from 
  \(b\) over \(\acl(ka)\cap\acl(kb)\).
  Conversely, suppose \(p(\U)\) is modular over~$k$, and let \(a\) be a tuple of realisations, and 
  \(K\) a difference field extending \(k\).
  By Lemma~\ref{lem:cbred}, 
  \(k\cdot\cb(a/K)=k(b)\) for some tuple~\(b\) of realisations of~\(p\).  
  Modularity tells us that~$a$ is independent from~$b$ over $\acl(ka)\cap\acl(kb)$,
  so that
  \(\cb(a/K)=\cb(a/k(b))\subseteq\acl(ka)\cap\acl(kb)\subseteq\acl(ka)\).
  as desired.
\end{remark}

To give a geometric characterisation of this property, we begin with some terminology around families of invariant subvarieties of a 
rational $\sigma$-variety.

\begin{definition}\label{def:rich}
Suppose $(V,\phi)$ is a rational $\sigma$-variety over~$(k,\sigma)$.
By a \emph{rich invariant family of subvarieties of $(V,\phi)$} we mean a rational $\sigma$-variety $(W,\psi)$ over~$k$ along with an irreducible invariant subvariety $Z\subseteq(V\times W,\phi\times\psi)$ such that:
\begin{itemize}
\item[(i)]
$Z$ projects dominantly onto both $V$ and $W$, with the general fibres of $Z\to W$ being absolutely irreducible subvarieties of~$V$,
\item[(ii)]
$\dim Z>\dim V$,
\item[(iii)]
\emph{Canonicity}: if $b\neq b'$ are general points of~$W$ then $Z_b\neq Z_{b'}$.\end{itemize}
\end{definition}

\begin{remark}
The general fibres of a rich family are infinite proper subvarieties of~$V$.
Indeed, canonicity forces them to be proper.
And if they were finite then they would be singletons (by absolute irreducibility), and together with canonicity this would mean that $Z$ determines a birational equivalence between~$V$ and~$W$, contradicting $\dim Z>\dim V$.
\end{remark}

The connection to one-basedness arises as follows:

\begin{lemma}\label{lem:richcb}
Suppose $Z\subseteq(V\times W,\phi\times\psi)$ is a rich invariant family of subvarieties of a rational $\sigma$-variety $(V,\phi)$ over~$k$.
Let $(a,b)$ be a generic point of $(Z,(\phi\times\psi)|_Z)$ over~$k$.
Then $\cb(a/k(b))\notin {k(a)}^{\alg}$.
\end{lemma}

\begin{proof}
Let $E$ be the minimal field of definition of the Zariski locus of $(a,\sigma(a))$ over $k(b)$.
By~\cite[Lemma~2.5]{qfint}, as $\qftp(a/k(b))$ is rational, $\cb(a/kb)$ is the difference field generated by~$E$.
It suffices to show, therefore, that $E\not\subseteq k(a)^{\alg}$.

Let $b_1,b_2$ be realisations of $\qftp(b/ka)$.
There is a difference field isomorphism $\alpha:(k(a,b_1),\sigma)\to(k(a,b_2),\sigma)$ that fixes $k(a)$ and takes $b_1$ to $b_2$, and which we can lift to a field-automorphism $\widehat\alpha$ of $\U$.
Then
$$\loc(a,\sigma(a)/k(b_1))^{\widehat\alpha}=\loc(a,\sigma(a)/k(b_2)).$$
Letting $E_i$ be the minimal field of definition of $\loc(a,\sigma(a)/k(b_i))$, for $i=1,2$, we get that $\widehat\alpha(E_1)=E_2$.
If $E_2=E_1$, then $\widehat\alpha$ must preserve $\loc(a,\sigma(a)/k(b_1))$, and so
$\loc(a,\sigma(a)/k(b_1))=\loc(a,\sigma(a)/k(b_2))$.
This in turn implies that
$\loc(a/k(b_1))=\loc(a/k(b_2))$.
As $\loc(a/k(b_i))$ is just $Z_{b_i}$, canonicity forces $b_1=b_2$.
So, if $b_1$ and $b_2$ were chosen to be distinct realisations of $\qftp(b/ka)$, then the corresponding $E_1,E_2$ are also distinct field-conjugates of $E$ over $k(a)$.
Since $\dim Z>\dim V$, we must have that $b\notin k(a)^{\alg}$, and hence $\qftp(b/ka)$ has infinitely many realisations, giving rise to infinitely many field-conjugates of $E$ over $k(a)$.
This shows that $E\not\subseteq k(a)^{\alg}$, as desired.
\end{proof}

We obtain the following geometric characterisation of one-basedness:

\begin{proposition}
\label{prop:1b-geo}
Suppose $(V,\phi)$ is a rational $\sigma$-variety over~$k$, and $p\in S_{\qf}(k)$ is its generic quantifier-free type.
    Then the following are equivalent:
    \begin{itemize}
        \item[(i)]
        $p$ is one-based.
        \item[(ii)]
        If $n\geq 1$ and $V_0$ is an irreducible invariant subvariety of $(V^n,\phi)$ over~$k$ that projects dominantly onto $V$ in each co-ordinate, then $(V_0,\phi)$ admits no rich families of invariant subvarieties.
        \end{itemize}
\end{proposition}

\begin{proof}
(i)$\implies$(ii).
Suppose $V_0$ is as in~(ii) and $Z\subseteq(V_0\times W,\phi\times\psi)$ is a rich invariant family of subvarieties.
Let $(a,b)$ be a generic point of $(Z,(\phi\times\psi)|_Z)$ over~$k$.
Since~$Z$ projects dominantly onto $V_0$ which in turn projects dominantly onto each copy of~$V$, and because these projections are equivariant, we have that~$a$ is an $n$-tuple of realisations of~$p$.
Letting $K:=k(b)$, Lemma~\ref{lem:richcb} tells us that $\cb(a/K)\notin k(a)^{\alg}$, showing that~$p$ is not one-based.

(ii)$\implies$(i).
Suppose $a$ is an $n$-tuple of realisations of $p$, and $K\supseteq k$ is a difference field extension such that $\cb(a/K)\notin k(a)^{\alg}$.
So~$a$ and~$K$ witness that~$p$ is not 1-based.
We may assume that $K=K^{\alg}$, using~\cite[Lemma~2.5(a)]{qfint}.
Let $V_0:=\loc(a/k)$, which is an irreducible invariant subvariety of $(V^n,\phi)$ projecting dominantly onto each~$V$ in each co-ordinate.
We aim to show that $(V_0,\phi)$ admits a rich family invariant of subvarieties.

Let $E:=k\cdot\cb(a/K)$.
By Lemma~\ref{lem:cbred}(a),
$E=k(b)=k\langle b\rangle$ for some finite tuple $b$ from~$K$.
In particular, $\qftp(b/k)$ is rational, and hence the generic quantifier-free type of a rational $\sigma$-variety $(W,\psi)$ over~$k$.
Let $Z=\loc(a,b/k)$.
This is an irreducible invariant subvariety of $(V_0\times W,\phi\times\psi)$ that projects dominantly onto~$V_0$ and~$W$, and will be our candidate for a rich family.

Observe first that $Z_b=\loc(a/K)$.
Indeed, as $\loc(a/K)$ is model-theoretically defined over the minimal field of definition of $\loc(a,\sigma(a)/K)$, and the latter is contained in $\cb(a/K)$, we have that $\loc(a/K)$ is defined over $k(b)$.
But $Z_b=\loc(a/k(b))$.
Hence, $Z_b\subseteq\loc(a/K)$.
On the other hand, as $k(b)\subseteq K$, we have that
$\loc(a/K)\subseteq\loc(a/k(b))=Z_b$, as desired.

In particular, as $K$ is algebraically closed, the general fibres of $Z\to W$ are absolutely irreducible.
Moreover, our assumption that $\cb(a/K)\notin k(a)^{\alg}$ implies that $k(b)\not\subseteq k(a)^{\alg}$, so we must have that $\dim Z>\dim V_0$.

So it remains to address canonicity. 
We first show that if $b'\models\qftp(b/k)$ and $Z_b=Z_{b'}$ then $b=b'$.
Indeed, let $\alpha:(k(b),\sigma)\to(k(b'),\sigma)$ be a $\sigma$-field-isomorphism over~$k$ such that $\alpha(b)=b'$, and
lift $\alpha$ to a field-automorphism $\widehat\alpha$ of~$\U$.
Then $Z_b^{\widehat\alpha}=Z_{b'}=Z_b$, from which it follows that  $\widehat\alpha$ is the identity on the minimal field of definition of $Z_b=\loc(a/K)$, which we denote by $F$.
Since $F\subseteq k(b)$, and $\widehat\alpha$ agrees with $\alpha$ on $k(b)$, we have that $\alpha$ is the identity on $F$.
As $\alpha$ preserves $\sigma$, it follows that $\alpha$ is the identity on $\langle k\cdot F\rangle$.
But by Lemma~\ref{lem:cbred}(b), the latter is $E$.
Hence $b'=\alpha(b)=b$, as desired.

Using the above we can, if necessary, modify the base of our family to make it canonical, without losing richness (namely $\dim Z>\dim V_0$).
Indeed, there exists an equivariant dominant rational map $\mu:(W,\psi)\dto(W',\psi')$ such that for general $b_1,b_2$ in~$W$, $\mu(b_1)=\mu(b_2)$ if and only if $Z_{b_1}=Z_{b_2}$.
This basically says that quotients exist in the category of rational $\sigma$-varieties; it can be deduced from elimination of imagianaries in~$\acf$, and an explicit proof can be found in~\cite[Proposition~3.2]{qfint}.
We obtain, thereby, an irreducible invariant subvariety $Z'\subseteq V_0\times W'$ with the same general fibres over~$W'$ as $Z\to W$, but now canonical.
The fact we proved in the previous paragraph, namely that  $Z_b\neq Z_{b'}$ whenever $b\neq b'\models\qftp(b/k)$, implies that $b\in\dcl(k\mu(b))\subseteq k(\mu(b))^{\alg}$, so that $\dim W=\dim W'$, and hence it is still the case that $\dim Z'>\dim V_0$.
\end{proof}

An example of a one-based rational type is the generic type of the rational dynamical system $(A,\phi)$ of  Example~\ref{example-sav}, where~$A$ is a simple abelian variety and $\phi\in \endo(A)$ is not of finite order.
This follows from~\cite{mm}, namely the property \emph{LMS} discussed there.

Of course,  the fixed field is not one-based; or rather, the generic type of $(\AA^1,\id)$ is not one-based.
Let us exhibit a rich invariant family of subvarieties of $(\AA^2,\id)$.
The parameter space (namely the $(W,\psi)$ of Definition~\ref{def:rich}) will be $(\AA^2,\id)$ as well, and we take $Z$ to be defined by $v_2=w_1v_1+w_2$ in co-ordinates $(v_1,v_2,w_1,w_2)$ for $\AA^2\times\AA^2$.
If we consider projection onto $(w_1,w_2)$ then the fibre of~$Z$ above $(c,d)$ is the line $v_2=cv_1+d$.
So $Z\to\AA^2$ is the family of all (non-vertical) lines in $\AA^2$.
It is easy to verify that all the conditions of being a rich family are satisfied: invariance (as the dymnamics are trivial all subvarieties are invariant), absolute irreducibility of the fibres, canonicity, and $\dim Z=3>\dim(\AA^2)$.

The following is a version of the Zilber dichotomy theorem for rank one types established in~\cite{acfa}, but adapted to the context of rational types.

\begin{theorem}[Zilber dichotomy for minimal rational types in characteristic 
  zero]\label{thm:zd}
If
$p\in S_{\qf}(k)$ is minimal and rational then either $p$ is one-based or there is a finite-to-one 
rational map $p\to q$ where $q\in S_{\qf}(k)$ is rational, $1$-dimensional, and qf-internal to~$\C$.
\end{theorem}

\begin{proof}
    Suppose $p$ is not one-based.
    We first show that $p\not\perp\C$.
    This does follow directly from results in~\cite{acfa}, namely, by combining Theorems~4.3 and~4.11 of that paper, and using the fact that our notion of ``one-based" agrees with their notion of ``modularity" (see our Remark~\ref{rem:modular}).
    However, we will follow the much simpler approach of~\cite{pillay-ziegler} and use instead the canonical base property for rational types, namely our Fact~\ref{fact:cbp} above.
    
        Non-one-basedness is witnessed by a tuple of realisations of~$p$, $a$, and a difference field extension $K\supseteq k$, such that $\cb(a/K)\not\in\acl(ka)$.
    We may take~$K$ to be algebraically closed.
    Let~$b$ be such that $k\cdot\cb(a/K)=k(b)$, this exists by 
    Lemma~\ref{lem:cbred}(a).
    Note that as $p$ is rational so is $\qftp(a/K)$.
    Fact~\ref{fact:cbp} therefore applies to $\qftp(a/K)$, and we have that $\qftp(b/ka)$ is qf--internal to~$\C$.
    So $b\in L(c)$ for some difference field extension $L\supseteq k(a)$ with $b\ind_{ka}L$ and $c$ a finite tuple from~$\C$.
    Since $b\notin\acl(ka)$, we must have that $b\notin\acl(L)$.
    Hence $b\nind_Lc$.   
    On the other hand, $b$ is in the field generated over~$k$ by a sequence of realisations of~$p$, this is the ``moreover" clause of~\ref{lem:cbred}(a) together with the fact that $a$ is a tuple of realisations of $p$.
    In particular, we can choose realisations $a_1,\dots,a_m$ of~$p$ such that $b\in\acl(La_1\dots a_m)$ and $a_{i+1}\notin\acl(La_1\dots a_i)$, for each $i<m$.
As $b\nind_Lc$, we must have $(a_1,\dots,a_m)\nind_Lc$, and so for some $i<m$ we have $a_{i+1}\nind_{L'}c$ where $L':=L(a_1,\dots,a_i)$.
But $a_{i+1}\ind_kL'$ because~$p$ is minimal and $a_{i+1}\notin\acl(L')$.
So $a_{i+1}\nind_{L'}c$ witnesses that $p\not\perp\C$.

Now, by~\cite[Proposition~2.11]{qfint}, that~$p$ is nonorthogonal to~$\C$ is witnessed by a rational map \(f:p\to q\) where $q\in S_{qf}(k)$ is a positive-dimensional rational type that is qf-internal to~\(\C\).
Minimality of~$p$ forces~$f$ to be finite-to-one.

It remains to show that~$q$ is $1$-dimensional.
(Note that as~$f$ need not be surjective, we don't even know {\em a prioiri} that~$q$ is minimal!)
Fix $a\models p$, so that $f(a)\models q$ and $\acl(ka)=\acl(kf(a))$.
As~$q$ is qf-internal to~$\C$ we have a difference field extension $L\supseteq k$ with $f(a)\ind_kL$, and a tuple~$c$ from $\C$ such that $L(f(a))=L(c)$.
Indeed, by definition we only have $f(a)\in L(c)$, but that we can choose~$L$ so that the other containment also holds follows from~\cite[Proposition~2.11]{qfint}.
We may also choose~$L$ such that $L\ind_{kf(a)}a$, so that $a\ind_kL$.
We now compute that
\begin{eqnarray*}
\dim(q)
&=&
\trdeg(f(a)/k)\ \ \ \text{ by definition,}\\
&=& \trdeg(f(a)/L)\ \ \ \text{ as $f(a)\ind_kL$,}\\
&=&
\trdeg(c/L)\ \ \ \text{ as $L(f(a))=L(c)$,}\\
&=&
SU(c/L)\ \ \ \text{ as $c$ is from the fixed field,}\\
&=&
SU(a/L)\ \ \ \text{ as $\acl(La)=\acl(Lf(a))=\acl(Lc)$,}\\
&=&
SU(a/k)\ \ \ \text{ as $a\ind_kL$,}\\
&=&
1\ \ \ \text{ as $p$ is minimal,}
\end{eqnarray*}
as desired.
\end{proof}

Here is an explicitly geometric formulation of the dichotomy:

\begin{corollary}
\label{cor:zd}
Suppose $(V,\phi)$ is a simple rational $\sigma$-variety over $(k,\sigma)$.
Then exactly one of the following holds:
\begin{itemize}
\item[(1)]
For any $n\geq 1$ and $V_0$ an irreducible invariant subvariety of $(V^n,\phi)$ over~$k$ that projects dominantly onto $V$ in each co-ordinate, $(V_0,\phi)$ admits no rich families of invariant subvarieties in the sense of Definition~\ref{def:rich} above.
\item[(2)]
$\dim V=1$ and there exists dominant $(V,\phi)\dto(V',\phi')$ over~$k$ such that $(V',\phi')$ is an isotrivial $\sigma$-curve.
\end{itemize}
\end{corollary}

\begin{proof}
We may assume that $\dim V>0$.
Let~$p$ be the generic quantifier-free type of $(V,\phi)$ over~$k$.
Then $p$ is minimal by simplicity (Proposition~\ref{prop:min-simple}).
Proposition~\ref{prop:1b-geo} tells us that one-basedness of~$p$ is equivalent to condition~$(1)$.
Assuming $p$ is not one-based, Theorem~\ref{thm:zd} tell us that there is a finite-to-one rational map $p\to q$ where $q$ is $1$-dimensional and qf-internal to~$\C$.
Letting $(V',\phi')$ be the rational $\sigma$-curve that $q$ is the generic of, we obtain a generically finite-to-one dominant rational map $(V,\phi)\dto(V',\phi')$.
It is shown in~\cite[Proposition~3.6]{qfint} that qf-internality of~$q$ to~$\C$ is equivalent to $(V',\phi')$ being isotrivial.
\end{proof}

Case~(1) admits a further analysis, carried out in~\cite[$\S$5.12]{acfa} and~\cite[$\S$4.1]{mm}, and splitting up further into the ``relationally trivial" case and a situation controlled by definable subgroups of simple abelian varieties.
But this further analysis requires one to work in the more general category of rational ``$\sigma$-correspondences" rather than $\sigma$-varieties, and we do not pursue it here.

\bigskip
\section{Exchange}
\label{sec:x}

\noindent
Much easier to verify than minimality is the following significant weakening that is motivated by the notion of ``no proper fibrations" introduced in~\cite{moosa-pillay2014}.

\begin{definition}
We will say that a rational type $p\in S_{\qf}(k)$ satisfies \emph{exchange} if the following holds: whenever $a\models p$ and $b\in\acl(ka)\setminus k$ with $\qftp(b/k)$ rational, then $a\in\acl(kb)$.
\end{definition}

Let us describe what this says about a rational $\sigma$-variety $(V,\phi)$ over~$k$.
Recall that $(V,\phi)$ is {\em primitive} if  whenever $(V,\phi)\dto(W,\psi)$ is a dominant equivariant rational map then $\dim W$ is either~$0$ or $\dim V$.
This was defined in the Introduction in the autonomous case, but it makes sense for rational $\sigma$-varieties in general.
We say that 
$(V,\phi)$ is {\em strongly primitive} if $(V',\phi')$ is primitive whenever $(V',\phi')$ admits a dominant equivariant rational map to~$(V,\phi)$ and $\dim(V')=\dim V$.
Exchange is strong primitivity of the corresponding rational $\sigma$-variety:

\begin{lemma}
\label{exchange-geom}
Suppose $p$ is the generic quantifier-free type of a rational $\sigma$-variety $(V,\phi)$ over~$k$.
Then $p$ satisfies exchange if and only if $(V,\phi)$ is strongly primitive.
\end{lemma}

\begin{proof}
Suppose $p$ satisfies exchange and we are give dominant equivariant rational maps $f:(V',\phi')\dto (V,\phi)$ and $g:(V',\phi')\dto(W,\psi)$ with $\dim (V')=\dim V$.
We need to show that $\dim W=0$ or $\dim W=\dim V$.
Let~$a'$ be a generic point of $(V',\phi')$ so that $a:=f(a)\models p$.
Then $b:=g(a')\in\dcl(ka')\subseteq \acl(ka)$.
Note that $\qftp(b/k)$ is rational as it is the generic quantifier-free type of $(W,\psi)$.
If $\dim W\neq 0$ then $b\notin\acl(k)$, and hence, by exchange, $a\in\acl(kb)$.
This forces $\dim W=\dim(b/k)=\dim(a/k)=\dim V$, as desired.

Conversely, suppose $(V,\phi)$ is strongly primitive, $a\models p$, and $b\in\acl(ka)\setminus k$ is such that $\qftp(b/k)$ is rational.
So $\qftp(b/k)$ is the generic quantifier-free type of a positive-dimensional rational $\sigma$-variety $(W,\psi)$.
Note that $\qftp(a,b/k)$ is also rational; it is the generic quantifier-free type of $(V',\phi')$ where $V'=\loc(a,b/k)$ (and so $\dim(V')=\dim V$).
The co-ordinate projections give us dominant equivariant rational maps $(V',\phi')\dto (V,\phi)$ and $(V',\phi')\dto(W,\psi)$.
By strong primitivity, as $\dim W>0$, we must have $\dim W=\dim V'$, which forces $a\in\acl(kb)$, as desired.
\end{proof}

That every rational minimal type satisfies exchange is clear: If $a\models p$ then $\su(a/k)=1$, and so if $b\in\acl(ka)\setminus k$ then $a\nind_kb$ and hence $\su(a/kb)=0$, which in turn forces $a\in\acl(kb)$.
But there are nonminimal types with exchange; a class of interesting examples is worked out in~$\S$\ref{xexample} below.

The following extends the Zilber dichotomy to all rational types satisfying exchange (not just minimal ones).
For complete finite rank types in stable theories this is essentially~\cite[Proposition~2.3]{moosa-pillay2014}.

\begin{theorem}
\label{thm:nonorth}
If $p\in S_{\qf}(k)$ is rational and satisfies exchange then exactly one of the following hold:
\begin{enumerate}
\item
$p$ is minimal and one-based, or
\item
there is a finite-to-one 
rational map $p\to q$ where $q\in S_{\qf}(k)$ is rational and qf-internal to~$\C$.
\end{enumerate}
\end{theorem}

\begin{proof}
Our proof is inspired by the main idea of~\cite[Proposition~6.1]{moosa-pillay2014}, which was about finite rank complete types in stable theories satisfying the canonical base property, and was itself inspired by an argument appearing in~\cite[Lemma~2.7]{COP} involving maximal coverings of algebraic dimension zero compact K\"ahler manifolds.
The canonical base property for $\acfa_0$ (Fact~\ref{fact:cbp}) is again a key ingredient.

If $p\not\perp\C$ then, by~\cite[Proposition~2.11]{qfint}, there is a rational map $f:p\to q$ where $q\in S_{\qf}(k)$ is a nonalgebraic rational type that is qf-internal to~$\C$.
Exchange forces~$f$ to be finite to-one.

So we may assume that $p\perp\C$.
We show that in this case~$p$ is minimal.
We then get one-basedness for free from the Zilber dichotomy (Theorem~\ref{thm:zd}).

Let $a\models p$ be arbitrary and set $V=\loc(a/k)$ to be the Zariski locus.
Note that, by rationality of~$p$ and absolute irreducibility of~$V$, for any difference field extension $K\supseteq k$, we have $a\ind_kK$ if and only if the $\loc(a/K)=V$.
It suffices to show, therefore, that whenever $K=K^{\alg}$  is a difference field extension of $k$, with $\loc(a/K)\neq V$, then $\dim\loc(a/K)=0$.
Choosing such $K$ so that $\dim\loc(a/K)$ is maximal, and setting $W:=\loc(a/K)\subsetneq V$, it suffices to prove that $\dim W=0$.
Let~$e$ be such that $k\cdot\cb(a/K)=k(e)=k\langle e\rangle$, as given by Lemma~\ref{lem:cbred}(a).
It suffices to show that $a\in\acl(ke)$.
Note that $e\notin k$ since $W\subsetneq V$, and that $\qftp(e/k)$ is rational.
Hence, by exchange, we reduce further to showing that $e\in\acl(ka)$.

Suppose, toward a contradiction, that $e\notin\acl(ka)$.
Consider
\[
E_0:=\acl(ke)\cap\acl(ka).
\]
We claim that $\qftp(a/E_0)\perp\C$.
As $\cb(a/K)\notin E_0$ we have, by~\cite[Lemma~2.5(a)]{qfint}, that 
$a\nind_{E_0}K$, and hence $a\nind_{E_0}e$.
Because $E_0$ and $E_0(e)$ are difference fields, it follows that
\[
V\supseteq \loc(a/E_0)\supsetneq \loc(a/E_0(e))=W,
\]
where the final equality uses the fact that $W$ is defined over $k(e)$, see for example Lemma~\ref{lem:cbred}(b).
As $\loc(a/E_0)$ is absolutely irreducible, the maximal choice of~$\dim W$ forces $\loc(a/E_0)=V$.
Hence $a\ind_kE_0$.
Since $p=\qftp(a/k)$ is $\C$-orthogonal, we conclude that 
$\qftp(a/E_0)\perp\C$.

On the other hand, the CBP, namely Fact~\ref{fact:cbp}, tells us that $\qftp(e/k(a))$ is qf-internal to $\C$.
In fact, the CBP was shown by Zo\'e Chatzidakis to yield a stronger conclusion, namely that the complete type of~$e$ over~$\acl(ke)\cap\acl(ka)$, so $\tp(e/E_0)$, is also {\em almost} $\C$-internal.
This is~\cite[Theorem~2.1]{zoe-cbp} in the general setting of supersimple theories, and Proposition~3.1 of that paper in the particular case of $\acfa_0$.
So, there is some difference field extension $F\supseteq E_0$ such that 
$e\ind_{E_0}F$ and $e\in\acl(F\C)$.
We can choose~$F$ so that $a\ind_{E_0e}F$.
Hence $a\ind_{E_0}F$.
Since $a\nind_{E_0}e$, this forces $a\nind_Fe$.
That $e\in\acl(F\C)$ therefore implies that $a\nind_F c$ for some finite tuple~$c$ from~$\C$.
This witnesses that $\qftp(a/E_0)\not\perp\C$, a contradiction.
\end{proof}

Here is the geometric formulation:

\begin{corollary}
\label{cor:zdexchange}
Suppose $(V,\phi)$ is a strongly primitive rational $\sigma$-variety over $(k,\sigma)$.
Then exactly one of the following holds:
\begin{itemize}
\item[(1)]
$(V,\phi)$ is simple and condition~(1) of Corollary~\ref{cor:zd} holds, or 
\item[(2)]
there exists a dominant equivariant rational map $(V,\phi)\dto(W,\psi)$ over~$k$ with $\dim W=\dim V$ and such that $(W,\psi)$ is isotrivial.
\end{itemize}
\end{corollary}

\begin{proof}
This is just the by now familiar algebraic dynamics translation of Theorem~\ref{thm:nonorth}, using the geometric characterisation of exchange given by Lemma~\ref{exchange-geom}.
\end{proof}

Theorem~\ref{thm:primitive} of the Introduction is the special case of this corollary when $k\subseteq\C$.

\bigskip
\section{Bounding nonminimality}
\label{sec:nmdeg}

\noindent
We continue to work over an algebraically closed inversive difference field $(k,\sigma)$ of characteristic zero.

The proof of (i)$\implies$(ii) in Proposition~\ref{prop:nonmin-geo} gives a bit more information:
if $p\in S_{\qf}(k)$ is a rational type that is not minimal then there are realisations $a_1,\dots,a_d$ of~$p$ such that~$p$ has a nonalgebraic forking extension to $S_{\qf}(ka_1,\dots,a_d)$.
This allows the following natural measure of nonminimality, extending the notion first introduced in~\cite{nmdeg} for finite rank types in stable theories:

\begin{definition}\label{def:nmdeg}
Suppose~$p\in S_{\qf}(k)$ is a nonalgebraic and nonminimal rational type.
Then the \emph{degree of nonminimality of~$p$}, denoted by $\nmdeg(p)$, is 
the least $d\geq 0$ for which there exist realisations $a_1,\dots,a_d$ of~$p$ 
such that~$p$ has a nonalgebraic forking extension to $S_{\qf}(ka_1,\dots,a_d)$.

If $p$ is algebraic or minimal then we set $\nmdeg(p):=0$.
\end{definition}

\begin{lemma}
\label{morley}
For $d>0$, any witness to $\nmdeg(p)=d$ is a realisation of $p^{(d)}$.
\end{lemma}

\begin{proof}
Suppose $a=(a_1,\dots,a_d)$ witnesses that $\nmdeg(p)=d$.
If $a_i\nind_k(a_1,\dots,a_{i-1})$ and $a_i\notin\acl(ka_1,\dots,a_{i-1})$ then  $\qftp(a_i/ka_1\dots a_{i-1})$ would witness that the degree of nonminimality is $<d$.
But if $a_i\in\acl(ka_1,\dots,a_{i-1})$ then we could drop it from the sequence, and what is left would still witness $\nmdeg(p)$, again contradicting that~$d$ is minimal.
So we must have $a_i\ind_k(a_1,\dots,a_{i-1})$ for each $i=1,\dots, d$, which implies that $a\models p^{(d)}$.
\end{proof}

\begin{lemma}\label{lem:exchange}
Suppose~$p\in S_{\qf}(k)$ is a nonalgebraic and nonminimal rational type.
If $\nmdeg(p)>1$ then $p$ satisfies exchange.
\end{lemma}

\begin{proof}
We follow the proof of Fact~2.2 in~\cite{nmdeg2}.

Suppose $a\models p$ and $b\in\acl(ka)\setminus k$.
Note that $a\nind_kb$.
Let $a'\models\tp(a/kb)$ such that $a'\ind_{kb}a$.
Then $a'\nind_k{b}$, and as $b\in\acl(ka)$, we must have that $a'\nind_ka$.
Hence $\qftp(a'/ka)$ is a forking extension of~$p$ to a single realisation of~$p$.
Since $\nmdeg(p)>1$, this forces $a'\in\acl(ka)$.
As $a'\ind_{kb}a$, it must be that $a'\in\acl(kb)$, and hence $a\in\acl(kb)$, as desired.
\end{proof}

\begin{proposition}\label{prop:redtowoc-nmdeg}
Suppose $p\in S_{\qf}(k)$ is a rational type.
Let $d>0$ be such that $p^{(d)}$ is not weakly orthogonal to~$\C$.
Then $\nmdeg(p)\leq d$.
\end{proposition}

\begin{proof}
The assumption that $p^{(d)}$ is not weakly orthogonal to~$\C$ already rules out~$p$ being algebraic.
We may also assume that $p$ is not minimal, else $\nmdeg(p)=0$ by convention and we are done.

Let $(V,\phi)$ be the positive-dimensional rational $\sigma$-variety over~$k$ that~$p$ is the generic quantifier-free type of.
As $p$ is not minimal, $\dim V>1$.

First consider the case of $d=1$, so that $p\not\perp^w\C$.
Then, by~\cite[Proposition~3.4(a)]{qfint},  $(V,\phi)$ admits a nonconstant invariant rational function.
That is, there is a rational map $f:p\to q$ where $q$ is the generic quantifier-free type of $(\AA^1,\id)$ over~$k$.
If $\nmdeg(p)>1$ then~$p$ satisfies exchange by Lemma~\ref{lem:exchange}.
Hence $f$ is finite-to-one, contradicting $\dim V>1$.

Now suppose $d\geq 2$ and $p^{(d)}\not\perp^w\C$.
If we let $a=(a_0,a_1)\models p^{(d)}$, where $a_0\models p$ and $a_1\models p^{(d-1)}$, then we see that $p':=\qftp(a_0/ka_1)$ is a nonforking extension of~$p$ that is not weakly orthogonal to~$\C$.
Hence, by the previous paragraph, $\nmdeg(p')\leq 1$.
This will be witnessed by some $a'\models p'$ and $q\in S_{\qf}({ka_1a'})$ a nonalgebraic forking extension of~$p'$.
But then~$q$ is a nonalgebraic forking extension of~$p$, and $(a',a_1)$ is a $d$-tuple of realisations of~$p$, so that $\nmdeg(p)\leq d$.
\end{proof}

In the original use of the above strategy, for $\dcf_0$, one obtained a slightly better bound of $d-1$ (assuming $d>1$), see~\cite[Proposition 7.1]{nmdeg}.
We do not know if that is the case here.
The difficulty is that the fibres of a rational map need not be absolutely irreducible, and hence need not have ``sharp" points.

Combining this with our earlier results in~\cite{qfint} on quantifier-free binding groups, we obtain the following bound on the degree of nonminimality.

\begin{theorem}
\label{thm:nmdeg}
Suppose $p\in S_{\qf}(k)$ is rational.
\begin{itemize}
\item[(a)]
$\nmdeg(p)\leq\dim(p)+3$.
\item[(b)]
If $k\subseteq\C$ then $\nmdeg(p)\leq 2$.
\end{itemize}
\end{theorem}

\begin{proof}
We may assume that~$p$ is not minimal nor algebraic.
We may also assume that $\nmdeg(p)>1$.
By Lemma~\ref{lem:exchange}, $p$ satisfies exchange.
As~$p$ is not minimal, Theorem~\ref{thm:nonorth} implies, in particular, that~$p$ is nonorthogonal to~$\C$.
Now we apply~\cite[Theorem~5.7]{qfint} which tells us that nonorthogonality to~$\C$ is witnessed by non-weak-orthgonality of $p^{(n+3)}$ to~$\C$, where $n=\dim(p)$.
Hence, Proposition~\ref{prop:redtowoc-nmdeg} implies that $\nmdeg(p)\leq n+3$.

In the autonomous case when $k\subseteq\C$, nonorthognality to~$\C$ is witnessed by non-weak-orthgonality of $p^{(2)}$, by \cite[Corollary~5.3]{qfint}, so that Proposition~\ref{prop:redtowoc-nmdeg} implies $\nmdeg(p)\leq 2$.
\end{proof}

What we actually expect, based on what was shown in~\cite{nmdeg2} for $\dcf_0$, is that $\nmdeg(p)\leq 2$ in general and $\nmdeg(p)\leq 1$ in the autonomous case.
The latter we would obtain automatically, if we could improve Proposition~\ref{prop:redtowoc-nmdeg} to a strict inequality.
But the former, namely an absolute bound of~$2$ for the degree of nonminimality in general, is further out of reach.

A non-autonomous rational type of $\nmdeg$~$2$ is exhibited in~$\S$\ref{sec-2trans} below.

For a geometric formulation of this theorem, it is useful to introduce the following refinement of the notion of simplicity (Definition~\ref{def:simple}).

\begin{definition}
\label{def:wsimple}
Suppose $(V,\phi)$ and $(W,\psi)$ are rational $\sigma$-varieties over~$k$.
We say that $(V,\phi)$ is {\em $(W,\psi)$-simple} if any proper irreducible invariant subvariety of $(V\times W,\phi\times\psi)$ projecting dominantly onto both~$V$ and~$W$ has dimension $\dim W$.
So $(V,\phi)$ is simple (according to Definition~\ref{def:simple}) if it is $(W,\psi)$-simple for all $(W,\psi)$.
\end{definition}

\begin{remark}
The proof of Proposition~\ref{prop:nonmin-geo} shows that if $(V,\phi)$ is not simple then it is not $(V^d,\phi)$-simple for some $d>0$.
(The action of $\phi$ on cartesian powers of~$V$ is taken to be co-ordinatewise.)
Note also that $(V^d,\phi)$-simplicity implies $(V^k,\phi)$-simplicity for $k\leq d$.
Indeed, if $Z\subseteq V^{k+1}$ witnesses the failure of $(V^k,\phi)$-simplicity then $Z\times V^{d-k}$ witnesses the failure of $(V^d,\phi)$-simplicity.
\end{remark}

We obtain the following geometric explanation of degree of nonminimality:

\begin{proposition}
\label{prop:nmdeg-geom}
Suppose $(V,\phi)$ is a rational $\sigma$-variety over~$k$ that is not simple, and $p\in S_{\qf}(k)$ is its generic quantifier-free type.
Then $\nmdeg(p)$ is the least positive integer $d$ such that $(V,\phi)$ is not $(V^d,\phi)$-simple.
\end{proposition}

\begin{proof}
We first show that if $(V,\phi)$ is not $(V^d,\phi)$-simple, then $\nmdeg(p)\leq d$.
Let $Z\subseteq V^{d+1}$ witness the failure of $(V^d,\phi)$-simplicity.
Let $(a_0,a_1,\dots,a_d)$ be a generic point of the induced rational $\sigma$-variety $(Z,\phi|_Z)$.
That $Z$ projects dominantly onto both the first $V$-co\"ordinate and the last $V^d$-co\"ordinates implies that $a_0\models p$ and $a:=(a_1,\dots,a_d)\models p^{(d)}$.
That $Z\neq V^{d+1}$ means that $(a_0,a)$ is not Zariski generic in $V^{d+1}$, and hence $a_0\nind_ka$ as $a$ is Zariski generic in $V^d$.
On the other hand, $\dim Z>\dim(V^d)=d\dim V$ forces $a_0\notin\acl(ka)$.
So $\qftp(a_0/ka)$ is a nonalgebraic forking extension of~$p$, witnessing that $\nmdeg(p)\leq d$.

It remains to show that if $\ell:=\nmdeg(p)$ then $(V,\phi)$ is not $(V^\ell,\phi)$-simple..
As $(V,\phi)$ is not simple, $p$ is neither algebraic nor miminal, and hence $\ell>0$ by convention.
By Lemma~\ref{morley}, there is $a\models p^{(\ell)}$ and $a_0\models p$ such that $a_0\nind_ka$ and $a_0\notin\acl(ka)$.
Let $Z=\loc(a_0,a/k)\subseteq V^{\ell+1}$.
Then~$Z$ is invariant for $(V^{\ell+1},\phi)$ since $\sigma(a_0,a)=\phi(a_0,a)$ because each of the $\ell+1$ co-ordinates of $(a_0,a)$ are realisations of~$p$.
Since $a_0\nind_ka$, we have that $Z\neq V^{\ell+1}$.
It projects dominantly onto~$V$ in the first co-ordinate because $a_0$ is Zariski generic in~$V$, and it projects dominantly onto $V^\ell$ on the last~$\ell$ co-ordinates because $a$, as it realises $p^{(\ell)}$, is Zariski generic in $V^\ell$.
Finally, as $a_0\notin\acl(ka)$, we have that
$$\dim Z=\trdeg(k(a_0,a)/k)>\trdeg(k(a)/k)=\dim(V^\ell).$$
So $(V,\phi)$ is not $(V^\ell,\phi)$-simple.
\end{proof}

Theorem~\ref{thm:nmdeg} translates into the following statement about algebraic dynamics:

\begin{corollary}
\label{cor:nmdeg-geom}
Suppose $(V,\phi)$ is a rational $\sigma$-variety over~$k$.
\begin{itemize}
\item[(a)]
$(V,\phi)$ is simple if and only if it is $(V^{n+3},\phi)$-simple where $n:=\dim V$.
\item[(b)]
Suppose~$\sigma$ is the identity on~$k$ so that $(V,\phi)$ is a rational dynamical system.
Then $(V,\phi)$ is simple if and only if it is $(V^{2},\phi)$-simple.
\end{itemize}
\end{corollary}

\begin{proof}
The left-to-right implications of both statements are clear.
For the converse, assume that $(V,\phi)$ is not simple.
Let $p\in S_{\qf}(k)$ be the generic quantifier-free type of $(V,\phi)$.
By Proposition~\ref{prop:nmdeg-geom},  $d:=\nmdeg(p)$ is the least positive integer such that $(V,\phi)$ is not $(V^d,\phi)$-simple.
Since $d\leq n+3$ by Theorem~\ref{thm:nmdeg}, we have that $(V,\phi)$ is not $(V^{n+3},\phi)$-simple.
In the case that $k\subseteq\C$, Theorem~\ref{thm:nmdeg} tells us that $d\leq 2$ and so $(V,\phi)$ is not $(V^2,\phi)$-simple.
\end{proof}

Theorem~\ref{thm:simple} of the Introduction is Corollary~\ref{cor:nmdeg-geom}(b).

\bigskip
\section{Bounding disintegration}
\label{sec:disintegration}
\noindent
A natural notion of being ``$m$-disintegrated" for complete types in stable theories would be to ask that any~$m$ distinct realisations be independent.
See,  for example, \cite[$\S$3]{c3} where this is studied in the case of differentially closed fields.
The naive analogue in our case, so for quantifier-free types in~$(\U,\sigma)\models\acfa_0$, does not make much sense.
For example, if $k\subseteq\C$ then only algebraic types $p\in S_{\qf}(k)$ could satisfy such a condition even for~$m=2$.
Indeed, for $p$ nonalgebraic, either every realisation is in the fixed field, in which case it is not hard to find a pair of distinct algebraically dependent realisations, or $a$ and $\sigma(a)$ are distinct dependent realisations of~$p$ for any $a\models p$.
So, at the very least, we should rule out such dependencies.

\begin{definition}
We say that $a,b\in\U^n$ are {\em $\sigma$-disjoint} if $a\neq\sigma^r(b)$ for any $r\in\ZZ$.
\end{definition}

\begin{remark}
\label{rem:disaut}
This notion of $\sigma$-disjointness is really only relevant in the (almost) autonomous case.
Indeed, if $(V,\phi)$ is a rational $\sigma$-variety over $k$ with generic quantifier-free type $p\in S_{\qf}(k)$ admitting a pair of distinct but not $\sigma$-disjoint realisations, then $(V,\phi)$ is defined over $k\cap\fix(\sigma^r)$ for some $r>0$.
\end{remark}

\begin{proof}
If $a\neq b$ are realisations of~$p$ that are not $\sigma$-disjoint then there is $r>0$ such that (without loss of generality) $a=\sigma^r(b)$.
Then $\sigma^r(b)$ is Zariski generic in both $V$ and $V^{\sigma^r}$ over~$k$, which forces $V=V^{\sigma^r}$, and hence~$V$ is defined over $k\cap\fix(\sigma^r)$.
On the other hand,
$$\phi(a)=\sigma(a)=\sigma^{r+1}(b)=\sigma^r(\phi(b))=\phi^{\sigma^r}(\sigma^r(b))=\phi^{\sigma^r}(a).$$
That is, $\phi$ and $\phi^{\sigma^r}$ agree on the generic point~$a$, so that $\phi=\phi^{\sigma^r}$, and hence~$\phi$ is defined over $k\cap\fix(\sigma^r)$ as well.
\end{proof}

\begin{definition}
Suppose $p\in S_{\qf}(k)$ and $m\geq 1$.
We say that~$p$ is {\em $m$-disintegrated} if every $m$ pairwise $\sigma$-disjoint realisations of~$p$ are independent over~$k$.
We say that~$p$ is {\em totally disintegrated} if it is $m$-disintegrated for all $m\geq 1$.
\end{definition}

\begin{remark}
\label{rem:disint}
Some immediate observations:
\begin{itemize}
\item[(a)]
$1$-disintegration is equivalent to~$p$ being nonalgebraic.
\item[(b)]
$2$-disintegration is equivalent to $\sigma$-disjointness agreeing with independence.
\item[(c)]
$m$-disintegrated implies $\ell$-disintegrated for $\ell\leq m$.
\item[(d)]
Totally disintegrated implies one-based.
\end{itemize}
\end{remark}

\begin{proof}
Only part~(d) requires proof.
It suffices to check that if $a=(a_1,\dots, a_n)$ and $b=(b_1,\dots, b_m)$ are finite tuples of relaisations of~$p$ then~\(a\) is independent from~\(b\) over \(\acl(ka)\cap\acl(kb)\)-- see Remark~\ref{rem:modular}.
As it does not change $\acl(ka)$ and $\acl(kb)$, we may assume that the $a_i$'s are pairwise $\sigma$-disjoint, and also the $b_i$'s.
If some $a_i,b_j$ is not $\sigma$-disjoint then $a_i\in \acl(ka)\cap\acl(kb)$, and hence it suffices to show that ~\(a'\) is independent from~\(b\) over \(\acl(ka)\cap\acl(kb)\) where $a'$ is the subtuple of $a$ obtained by removing such $a_i$'s.
But now, the full sequence of co-ordinates of $a'$ and $b$ are pairwise $\sigma$-disjoint, and hence independent over~$k$ by total disintegratedness, which in particular implies ~\(a'\) is independent from~\(b\) over \(\acl(ka)\cap\acl(kb)\), as desired.
\end{proof}

\begin{example}
The following is a special case of~\cite[Lemma~6.1]{acfa}:
\begin{itemize}
\item[$(\dagger)$]
for any algebraically closed inversive difference field $(E,\sigma)$, 
if $\sigma(a)=a^2+1$ and $\sigma(b)=b^2+1$ with $b\in\acl(Ea)\setminus E$ then $a,b$ are not $\sigma$-disjoint.
\end{itemize}
It follows that if $p$ is the generic quantifier-free type over~$k$ of the rational dynamics on the affine line given by 
$x\mapsto x^2+1$ then~$p$ is totally disintegrated.
Indeed, if $a_1,\dots,a_m$ is a dependent set of realisation of~$p$, let $i<j$ be such that $a_i\nind_{ka_1\dots a_{j-1}}a_j$, and set $E:=\acl(ka_1\dots a_{j-1})$.
Minimality of~$p$ implies  $a_i\in\acl(Ea_j)\setminus E$, and so $a_i,a_{j}$ are not $\sigma$-disjoint by~$(\dagger)$.
\end{example}

We aim to prove:

\begin{theorem}\label{d3}
Suppose $p\in S_{\qf}(k)$ is rational and nonalgebraic.
\begin{itemize}
\item[(a)]
Suppose $\dim(p)=1$.
If~$p$ is $4$-disintegrated then it is totally disintegrated.
\item[(b)]
 Suppose $\dim(p)>1$ and let
$\ell:=\max\{3, \nmdeg(p)+1\}$.
If~$p$ is $\ell$-disintegrated then it is totally disintegrated.
\item[(c)]
Suppose $k\subset\C$.
If $p$ is $3$-disintegrated then it is totally disintegrated.
\end{itemize}
\end{theorem}

Combining parts~(a) and~(b) of Theorem~\ref{d3} with Theorem~\ref{thm:nmdeg}(a) will give us that $(\dim(p)+4)$-disintegration does imply total disintegration in general.
So, along with part~(c), this will establish Theorem~\ref{d3-intro} of the Introduction.

We do not expect Theorem~\ref{d3} to be sharp; indeed, we expect an absolute bound independent of~$p$, even in the nonautonomous case.
However, the example constructed in~$\S$\ref{sec-2trans} does show that such an absolute bound would have to be at least~$3$.
In the autonomous case we do expect~$3$ to be sharp bound: there should exist a rational type that is $2$-disintegrated but not $3$-disintegrated.
But we do not yet have an example.
(As the proof will show, such an example must be minimal and one-based, and related to the simple abelian variety dynamics of Example~\ref{example-sav}, but requiring an additional parameter to witness that connection.)

We begin by working toward a proof of part~(b) of Theorem~\ref{d3}, which deals with the most general case.
We require the following two preliminary propositions, which may be of interest in their own right.

\begin{proposition}
\label{dismin}
Suppose $p\in S_{\qf}(k)$ is rational and nonalgebraic.
If $p$ is $m$-disintegrated for $m>\nmdeg(p)$ then $p$ is minimal.
\end{proposition}

\begin{proof}
Suppose $p$ is not minimal and let $m>d:=\nmdeg(p)>0$.
Let $(V,\phi)$ be the rational $\sigma$-variety that $p$ is the generic quantifier-free type of.
So $(V,\phi)$ is not simple.
By Proposition~\ref{prop:nmdeg-geom}, $(V,\phi)$ is not $(V^{m-1},\phi)$-simple.
Let $Z\subseteq V^m$ be a witness.
That is, $Z$ is a proper irreducible invariant subvariety of $(V^m,\phi)$ over~$k$ that project dominantly on the first $V$-co\"ordinate and the last $V^{m-1}$-co\"ordinates, and $\dim Z>(m-1)\dim V$.
A generic point of $(Z,\phi)$ will be of the form $(a_1,\dots,a_m)$ where each $a_i\models p$.
If $a_i=\sigma^r(a_j)$, for some $i\neq j$ and $r\in\ZZ$, then, possibly after reordering, we may assume that $i=1$, $j>1$, and $r\geq 0$.
So
$$k(a_1,\dots,a_m)\subseteq k\langle a_2,\dots,a_m\rangle=k(a_2,\dots,a_m).$$
But this contradicts $\dim Z>(m-1)\dim V$.
Hence $a_1,\dots,a_m$ is pairwise $\sigma$-disjoint.
On the other hand, $a_1,\dots, a_m$ is not independent as $Z\subsetneq V^m$.
So~$p$ is not $m$-disintegrated.
\end{proof}

\begin{proposition}\label{min1based-dis}
Suppose~$p\in S_{\qf}(k)$ is a minimal one-based type.
If~$p$ is $3$-disintegrated then it is totally disintegrated.
\end{proposition}

\begin{proof}
One possibility is that $p$ is ``trivial" in the sense that any set of pairwise independent realisations is independent.
See~\cite[Definition~3.1(5)]{acfa}.
In that case, if~$p$ is $2$-disintegrated, any set of pairwise $\sigma$-disjoint realisations will be independent.
Hence, already $2$-disintegrated implies totally disintegrated.
(This does not use minimality.)

It suffices to show, then, that if $p$ is $3$-disintegrated then it is trivial.
This will follow from the characterisation of minimal one-based nontrivial types by Chatzidakis and Hrushovski in~\cite[Theorem~5.12]{acfa}.
From nontriviality of~$p$, the proof of that theorem produces, in particular,
a commutative algebraic group~$A$ over~$k$,
and a pair of minimal quantifier-free types $r,s\in S_{\qf}(k)$ extending~$A$ such that:
\begin{itemize}
\item[(i)]
there are independent $b_1,b_2\models s$ such that $r=\qftp(b_1-b_2/k)$, and
\item[(ii)]
$s$ is interalgebraic with $p$.
\end{itemize}
We will show that this implies~$p$ is not $3$-disintegrated.
First, observe:

\begin{claim}
If $h\models r$ and $b\models s$ are independent over~$k$ then $h+b$ is a realisation of~$s$ that is independent from both~$h$ and~$b$ over~$k$.
\end{claim}

\begin{claimproof}
Note that as~$k$ is algebraically closed, $\qftp(h/k)\cup \qftp(b/k)$ and $h\ind_kb$ determines $\qftp(h,b/k)$.
Hence it suffices to prove the claim for {\em some} independent $h\models r$ and $b\models s$.

Let $h:=b_1-b_2$ and $b:=b_2$.
To see that $h\ind_kb$, note that
\begin{eqnarray*}
SU(h/kb)
&=&
SU(b_1/kb)\ \ \text{ as $h=b_1-b$}\\
&=&
SU(b_1/k)\ \ \text{ as $b_1\ind_kb$}\\
&=&
SU(s)\\
&=&
SU(r)\ \ \text{ as both are~$1$}\\
&=&
SU(h/k).
\end{eqnarray*}
The same argument, with the roles of $b_1$ and~$b$ interchanged, shows that $h\ind_kb_1$.
So $h\ind_kh+b$.
That $b\ind_kh+b$ is just the fact that $b_2\ind_kb_1$.
\end{claimproof}

Now fix $h\models r$ and $b\models s$ independent over~$k$.
Then, by the Claim, $b$, $h+b$, and $2h+b$ are realisations of~$s$ that are pairwise independent but dependent as a triple.
As~$p$ is interalgebraic with~$s$, it follows that there exist realisations $a_1,a_2,a_3$ of $p$ that are pairwise independent but dependent, over~$k$.
In particular, $a_1,a_2,a_3$ are pairwise $\sigma$-disjoint.
So this witnesses that~$p$ is not $3$-disintegrated.
\end{proof}

We can dow deduce part~(b) of Theorem~\ref{d3}.

\begin{proof}[Proof of Theorem~\ref{d3}(b)]
Suppose $p\in S_{\qf}(k)$ is a rational nonalgebraic type that is $\ell$-disintegrated, where $\ell:=\max\{3, \nmdeg(p)+1\}$.
Because $\nmdeg(p)<\ell$, Proposition~\ref{dismin} implies that $p$ is minimal.
Now, by the Zilber Dichotomy (Theorem~\ref{thm:zd}), $p$ is either one-based or $1$-dimensional.
So, under the assumption that $\dim (p)>1$ we conclude that $p$ is one-based, and Proposition~\ref{min1based-dis} applies.
As $\ell\geq 3$, we conclude that~$p$ is totally disintegrated, as desired.
\end{proof}

Next we consider the autonomous case; so we aim to prove Theorem~\ref{d3}(c).
Note that we need only consider the $1$-dimensional case, as otherwise the result follows from part~(b) along with Theorem~\ref{thm:nmdeg}(b) which tell us that $\nmdeg(p)\leq 2$.
However, we will actually work in arbitrary dimension as it is no more difficult, and the preliminary work may be of general interest.

The first step is to show that autonomous rational types nonorthogonal to the fixed field are never even $2$-disintegrated.
This will be Proposition~\ref{autiso2dis} below, but we begin with a preparatory lemma about pseudofinite fields stated at a level of generality that will also be useful later.

\begin{lemma}
\label{psf-not2dis}
Suppose \(V\) is a positive-dimensional irreducible variety 
over~\(k\), and \(D\subseteq{}V(\C)\) is a \(\kk\)-definable subset that 
is Zariski dense in~\(V\).
Then there exist \(a,b \in{}D\) such that
\begin{itemize}
\item[(i)]
 $k(a)\not\subseteq k(b)$ and $k(b)\not\subseteq k(a)$,
 \item[(ii)]
 $a$ and $b$ are Zariski generic in~\(V\) over~\(\kk\), and
 \item[(iii)]
 \(\Set{a,b}\) is algebraically dependent over~\(\kk\).
 \end{itemize}
\end{lemma}

\begin{proof}
	Note that, since \(\C\) is stably embedded in \(\U\), the subset \(D\subseteq V(\C)\) is defined over \(C_\kk:=k\cap\C\).
	Also, as $V$ has a Zariski dense subset of $\C$-points, it is also defined over $C_k$.
	Hence, this is a statement entirely about the pseudofinite 
	field \((\C,+,\times)\).
	We point the reader to~\cite{zoe-psf} for a survey 
	of the model theory of pseudo-finite fields.
	
	By quantifier-reduction in pseudofinite fields, \(D\) will be a finite 
	union of sets of the form \(g(E)\) where \(g:W\to{}V\) is a generically 
	finite-to-one rational map from an irreducible algebraic 
	variety~\(W\) over~$k$, and \(E\subseteq W(\C)\) is a Zariski dense subset of~$W$ that is
	quantifier-free definable in \((\C,+,\times)\) over \(C_k\).
	One of these \(g(E)\) will then also be Zariski dense in~\(V\), and we fix 
	it.
	It now suffices to find \(u,v\in{}W(\C)\) such that
\begin{itemize}
\item $k(g(u))\not\subseteq k(g(v))$ and $k(g(v))\not\subseteq k(g(u))$,
 \item
 $u$ and $v$ are Zariski generic in~\(W\) over~\(\kk\), and
 \item
 \(\Set{u,v}\) is algebraically dependent over~\(\kk\).
 \end{itemize}
	Indeed, such \(u,v\) will be in \(E\) and $a:=g(u), b:=g(v)\in D$ will satisfy (i)--(iii).
	
	To that end, let \(Z\subseteq W\times W\) be a self-correspondence of degree greater than $\deg(g)$.
	That is, $Z$ is an absolutely irreducible subvariety over~\(C_\kk\) such that both co-ordinate projections 
	onto~\(W\) are dominant, finite-to-one, and of degree strictly greater that \(\deg(g)\).
	That such 
	subvarieties always exist, of every degree, is not difficult to verify by, 
	for example, reducing to the case that \(W\) is affine space and considering power functions.

As~\(\C\) is pseudo-algebraically-closed there is \((u,v)\in{}Z(\C)\) Zariski generic in~$Z$ over~$C_k$.
Note that $(u,v)$ is in fact Zariski generic over~$k$, as if it is contained in a subvariety defined over $k$ then it would be contained in the Zariski closure of the $\C$-points of that subvariety which is defined over $C_k$.
So $u,v\in W(\C)$ are Zariski generic in~$W$.
If $k(g(v))\subseteq k(g(u))$ then
	$$[k(u,v):k(u)]\leq [k(v):k(g(u))]\leq[k(v):k(g(v))]\leq\deg(g),$$
contradicting our choice of~$Z$.
	Similarly, $k(g(u))\not\subseteq k(g(v))$.
	From $Z\neq W\times W$ we deduce that~$u$ and~$v$ are algebraically dependent over~$k$.
\end{proof}

\begin{proposition}
\label{autiso2dis}
Suppose $k\subseteq\C$ and $p\not\perp\C$ is a rational type over~$k$.
Then $p$ is not $2$-disintegrated.
\end{proposition}

\begin{proof}
We first claim that, as we are in the autonomous case, $m$-disintegratedness is a birational invariant.
Indeed, this follows once one observes that if $f:p\to p'$ is a birational map over~$k$, and $a,b $ are realisations of~$p$, then $a, b$ are $\sigma$-disjoint if and only if $f(a),f(b)$ are $\sigma$-disjoint.
This follows in turn from the fact that $\sigma$ commutes with~$f$ (and $f^{-1}$) as $k\subseteq\C$.
So, in proving that $p$ is not $2$-disintegrated, we can freely replace~$p$ by a rational type that is birational to~$p$.

Next,  \cite[Proposition~2.12]{qfint} tells us that nonorthogonality to~$\C$ is witnessed by a rational map $p\to q$ with~$q$ nonalgebraic and qf-internal to~$\C$.
If $a\models p$ and $f^{-1}(f(a))$ is infinite, then we can find $b\models p$ that is $\sigma$-disjoint from~$a$ but with $f(b)=f(a)$.
The latter forces $a\nind_kb$, and we have a witness to the failure of $2$-disintegration.
We may therefore assume that~$f$ is finite-to-one.

We now use a result from~\cite{bms}.
Let $(V,\phi)$ be the rational dynamical system over~$k$ whose generic quantifier-free type is~$p$.
The existence of finite-to-one $f:p\to q$ with $q$ qf-internal to~$\C$ implies, in the language of~\cite{bms}, that the {\em stabilised algebraic dimension} of $(V,\phi)$ is $\dim V$.
Corollary~A of~\cite{bms} then tells us that $(V,\phi)$ is, after possibly replacing $(V,\phi)$ by a birationally equivalent rational dynamical system, {\em translational}; meaning that there is a faithful algebraic group action, $G\times V\to V$ over~$k$, such that~$\phi$ agrees with the action on~$V$ of some $k$-point $g\in G(k)$.

If~$g$ is of finite order in~$G$ then~$\phi$ is of finite order.
By quotienting out by the action of the finite group generated by~$\phi$, we obtain a dominant equivariant rational map to trivial dynamics, $f:(V,\phi)\dto(V',\id)$, with $\dim V'=\dim V$.
(We cannot directly lift non-2-disintegration from $(V',\id)$, where it is easily seen to be true, because the induced rational map on the generic types may not be surjective.)
Note that a pair of realisations of~$p$ are $\sigma$-disjoint if and only if they have distinct images in $V'$.
Moreover, they are dependent over~$k$ if and only if their images are.
So it suffices to find a pair of realisations of~$p$ whose image in $V'$ are distinct but dependent.
This follows from applying Lemma~\ref{psf-not2dis} to $V'$ with 
$D:= f\big((V,\phi)^\sharp\big)$.
Note that~$D$ is Zariski dense in $V'$, and $f(p(\U))$ is precisely the set of points in~$D$ that are Zariski generic in~$V'$ over~$k$.
So Lemma~\ref{psf-not2dis} does indeed do the job for us, implying that~$p$ is not $2$-disintegrated in this case.

So we may assume that~$g$ is not of finite order.
Then the Zariski closure, $H$, of the cyclic subgroup of~$G$ generated by~$g$, is a commutative algebraic subgroup of~$G$ over~$k$ that is positive-dimensional.
Hence, there is some $h\in H(k)$ that is not equal to any integer power of~$g$.
Fix $a\models p$ and let $b:=h(a)$.
We claim that $b\models p$ too.
Indeed, $b$ is Zariski generic in~$V$ over~$k$ (as $h$ is an automorphism of~$V$) and
$$\sigma(b)=\sigma(h(a))=h(\sigma(a))=h(\phi(a))=h(g(a))=g(h(a))=\phi(b)
$$
as required.
The fact that $h$ is not an integer power of $g$, and $g$ agrees with $\sigma$ on~$a$ and~$b$, implies that $a, b$ are $\sigma$-disjoint.
But they are visibly dependent over~$k$.
Hence, $p$ is not $2$-disintegrated.
\end{proof}

We can now prove Theorem~\ref{d3}(c).

\begin{proof}[Proof of Theorem~\ref{d3}(c)]
Suppose that $k\subseteq\C$ and $p$ is $3$-disintegrated.
By Theorem~\ref{thm:nmdeg}(b), $\nmdeg(p)< 3$ in this case, so that Proposition~\ref{dismin} applies and we have that $p$ is minimal.
If $p$ is one-based then $3$-disintegratedness implies total disintegration by Proposition~\ref{min1based-dis}, and we are done.
By the Zilber Dichotomy, the only other possibility is that~$p$ is nonorthogonal to~$\C$.
But already $2$-disintegratedness rules out this possibility by Proposition~\ref{autiso2dis}.
\end{proof}

Finally, we address the possibly nonautonomous $1$-dimensional case.
Again, the only situation left to deal with is when the type is nonorthogonal to the fixed field.

\begin{proposition}
\label{nonautiso4dis}
Suppose $p\not\perp\C$ is a $1$-dimensional rational type over~$k$.
Then $p$ is not $4$-disintegrated.
\end{proposition}

\begin{proof}
Let $(V,\phi)$ be a rational $\sigma$-curve over $(k,\sigma)$ whose generic type is~$p$.
Since each birationality class of a curve is represented by a unique smooth 
and projective curve, we assume that~\(V\) is such.
Likewise, a rational map 
between two such curves is (represented by) a regular one, so that $\phi:V\to V^\sigma$ can be taken to be regular.

Next we observe that, in the $1$-dimensional case, $p\not\perp\C$ implies that $\phi$ is an isomorphism.
Indeed, letting \(a\models p\) we have that
$[k(a):k(\sigma^\ell(a)]=\deg(\phi)^\ell$, so that if $\deg(\phi)>1$ then $p$ is ``unbounded" in the terminology of~\cite{acfa}, and Theorem~4.8 of that paper implies that $1$-dimensional unbounded types are orthogonal to the fixed field.
So $\phi:V\to V^\sigma$ is an isomorphism.

We now consider the genus of the curve $V$.
We will show that if the genus of~\(V\) is less than~\(2\) then 
	\(p\) is not \(4\)-disintegrated, and otherwise it is not even \(2\)-disintegrated.
	We denote by $C_k:=\C\cap k$ the fixed field of $(k,\sigma)$.

Case~$1$: $V$ is of genus~$0$.
So \(V=\PP^1\), and \(\phi\in\pgl_2(\kk)\). Given 
		distinct \(u,v,w\in{}V(\U)\), let \(A_{u,v,w}\in\pgl_2(\U)\) be the 
		unique element mapping \(u,v,w\) to \(0,1,\infty\), respectively. Note 
		that \(A_{\phi(u),\phi(v),\phi(w)}\circ\phi=A_{u,v,w}\), and that 
		\({A_{u,v,w}}^\sigma=A_{\sigma(u),\sigma(v),\sigma(w)}\), so if 
		\(u,v,w\in\shp{\pp{V,\phi}}(\U)\), then \(A_{u,v,w}\) maps 
		\(\shp{\pp{V,\phi}}(\U)\) bijectively to \(V(\C)\).
		Now fix \(u,v,w\) independent realisations of \(p\), and let
		\(A=A_{u,v,w}\).
		Hence, if \(t\in\PP^1(C_\kk)\), then \(x=A^{-1}t\) is another 
		realisation of \(p\), which is visibly dependent on \(u,v,w\).
		Hence, to show that \(p\) is not \(4\)-disintegrated, it thus suffices to show 
		that we may choose \(t\) so that \(x\) is \(\sigma\)-disjoint from 
		\(u,v,w\).
		This is clear if \(C_\kk\) is uncountable, but we prove it in general. 
		By 
		Remark~\ref{rem:disaut}, we may assume that \(\phi^{\sigma^r}=\phi\) for some \(r>0\).  
		If \(x\) is not \(\sigma\)-disjoint from \(u\), we have, for some 
		\(m>0\), that \(x=\sigma^mu\), i.e., that   
		\(t=A\phi^{\sigma^{m-1}}\circ\dots\circ\phi\circ{}A^{-1}0\) (and 
		similarly for \(v,w\), replacing \(0\) with \(1,\infty\)).  Since 
		\(\phi^{\sigma^r}=\phi\), all such twisted powers of \(\phi\) (as well as 
		the conjugation by \(A\)) are defined over the same finitely generated 
		extension of \(\QQ\). Hence, any \(t\) that does not belong to this 
		finitely generated extension will yield \(x=A^{-1}t\) which is  
		\(\sigma\)-disjoint from \(u,v,w\).
		So it remains to observe that 
		\(C_\kk\) itself is not finitely generated.
		But if \(C_\kk\) were finitely generated, then, as $k$ is algebraically closed, the fixed field of \(\sigma|_{\QQ^{\alg}}\) would also be finitely generated, implying the absurdity that the absolute Galois group of 
	 \(\QQ\) is cyclic-by-finite.

Case~$2$: $V$ is of~genus~\(1\).
Then \(V\) is an elliptic curve, and 
		\(\phi:V\to{}V^\sigma\) is an isomorphism of elliptic curves, so there is 
		a group isomorphism \(\phi_0:V\ra{}V^\sigma\) and an element 
		\(a\in{}V^\sigma(k)\) such that \(\phi(y)=\phi_0(y)+a\) for all \(y\in{}V\).  
		Hence, \(\shp{\pp{V,\phi}}(\U)\) is a torsor over the group 
		\(\shp{\pp{V,\phi_0}}(\U)\), and this also implies that \(p\) is not \(4\)-disintegrated: if 
		\(x,y,z\) are any independent realisations of~\(p\), then \(x+y-z\) is a 
		realisation as well, and these are~\(4\) algebraically dependent 
		realisations.
		It is easy to see that independence of $x,y,z$ implies $\sigma$-disjointness of $x,y,z,x+y-z$.
		
Case~$3$: Higher genus.
Suppose, first of all, that there are distinct realisations of~$p$ that are not $\sigma$-disjoint.
In this case we prove that $p$ is not $2$-disintegrated.
By Remark~\ref{rem:disaut}, $(V,\phi)$ is defined over $k_0:=k\cap\fix(\sigma^r)$ for some $r>0$.
Hence $V^{\sigma^r}=V$ and \(\phi^{(r)}:=\phi^{\sigma^{r-1}}\circ\dots\circ\phi\) is an automorphism of~$V$.
The higher genus assumption implies that \(\phi^{(r)}\) has finite order.
It follows that $\phi^{(\ell)}=\id_V$ for some $\ell>0$.
So, if $a\in (V,\phi)^\sharp(\U)$ then $\sigma^\ell(a)=\phi^{(\ell)}(a)=a$.
That is, $D:=(V,\phi)^\sharp(\U)\subseteq V\big(\fix(\sigma^\ell)\big)$.
Applying Lemma~\ref{psf-not2dis}, but with $\sigma^\ell$, we find $a,b\in D$ such that
\begin{itemize}
\item[(i)]
 $k(a)\not\subseteq k(b)$ and $k(b)\not\subseteq k(a)$,
 \item[(ii)]
 $a$ and $b$ are Zariski generic in~\(V\) over~\(\kk\), and
 \item[(iii)]
 \(\Set{a,b}\) is algebraically dependent over~\(\kk\).
 \end{itemize}
 If $a=\sigma^m(b)$ for some $m\geq 0$ then $a=\phi^{(m)}(b)\in k(b)$, contradicting~(i).
 Similarly, $b\neq\sigma^m(a)$ for any $m\geq 0$, and we have that $a,b$ are $\sigma$-disjoint.
 By~(ii), and the fact that $D=(V,\phi)^\sharp(\U)$, we have that $a, b$ realise~$p$.
By~(iii), they witness the failure of $2$-disintegration.
				
We may therefore assume that distinct realisations of~$p$ are always $\sigma$-disjoint; so that what we seek is a pair of distinct but dependent realisations of~$p$.
By \cite[Proposition~2.12]{qfint}, nonorthogonality to~$\C$ is witnessed by a rational map $p\to q$ with~$q$ nonalgebraic and qf-internal to~$\C$.
Letting $(V',\phi')$ be an isotrivial rational $\sigma$-curve whose generic type is~$q$, and again taking~$V'$ to be a smooth projective model, we obtain a surjective equivariant map
$f:(V,\phi)\to(V',\phi')$.
By \cite[Theorem~1.6]{qfint}, the quantifier-free binding group of $q$ relative to~$\C$, $\aut_k(q/\C)$, acts by birational transformations, and hence by automorphosms, on~$V'$.
See the discussion in~$\S$\ref{subsect:nonorth} above.
The higher genus assumption thus implies that $\aut_k(q/\C)$ is finite.
Let $g:V'\to W$ be the quotient by the action of this finite group of automorphisms.
Now, two realisations, $a$ and $b$ of~$q$, are in the same $\aut_k(q/\C)$-orbit precisely if they have the same qf-type over \(k(\C)\).
Moreover, by~\cite[Proposition~2.5]{qfint}, $\qftp(a/k(\C))$ is isolated by $\qftp(a/k, k(a)\cap\C)$.
It follows that there are  invariant rational functions $g_1,\dots,g_n\in k(V')$ such that $g(a)=g(b)$ if and only if $g_i(a)=g_i(b)$ for all~$i$.
But this means that, up to birational equivalence, we can take $W$ to be a smooth projective curve over~$C_k$ and $g:(V',\phi')\to (W,\id)$ to be equivariant.
Composing, we obtain a surjective equivariant map $h:=g\circ f:(V,\phi)\to(W,\id)$.
So the image under~$h$ of Let \((V,\phi)^\sharp(\U)\) is a definable subset \(D\subseteq{}W(\C)\) that is Zariski generic in~$W$.
By Lemma~\ref{psf-not2dis}, $D$ includes a pair of 
		distinct algebraically dependent generic points of~$W$.
		The preimages of these in \((V,\phi)^\sharp(\U)\) will be distinct realisations of~$p$ that are dependent.
\end{proof}

And with that we can complete the proof of Theorem~\ref{d3}.

\begin{proof}[Proof of Theorem~\ref{d3}(a)]
Suppose that $p\in S_{\qf}(k)$ is rational and $\dim(p)=1$.
In particular, of course, $p$ is minimal.
If $p$ is one-based then already $3$-disintegratedness implies totally disintegrated by Proposition~\ref{min1based-dis}.
So we may assume that \(p\not\perp\C\), in which case Proposition~\ref{nonautiso4dis} implies non-$4$-disintegration.
\end{proof}

\medskip
\subsection{Proof of Theorem~\ref{c3thm-intro}}
We can now deduce our main transformal transcendence result.
Here it is, restated to include the possibly nonautonomous case:

\begin{corollary}\label{c3thm}
Consider the difference-algebraic equation
\begin{equation}
\label{de}
\sigma^n(y)=f\big(y, \sigma(y),\dots,\sigma^{n-1}(y)\big)\tag{E}
\end{equation}
of order $n\geq 1$, where~$y$ is a single variable and~$f\in k(y_0,\dots,y_{n-1})$.
For each $m\geq 1$, consider the following condition on~$(\ref{de})$:
\begin{itemize}
\item[$(C_m)$]
For any pairwise $\sigma$-disjoint solutions $b_1,\dots,b_m$ of~$(\ref{de})$, the transendence degree of $k\langle b_1,\dots,b_m\rangle$ over~$k$ is~$mn$.
\end{itemize}
Then $(C_{n+4})\implies(C_m)$ for all~$m$.
If $k\subseteq\C$ then $(C_3)\implies(C_m)$ for all~$m$.
\end{corollary}

\begin{proof}
We associate to~(\ref{de}) the $\sigma$-rational variety $(\AA^n, \phi)$ where
$$\phi(y_0,\dots,y_{n-1})=\big(y_1,\dots, y_{n-1}, f(y_0,\dots,y_{n-1})\big).$$
Let $p$ be the generic quantifier-free type of $(\AA^n,\phi)$.
Note that $a=(a_0,a_1,\dots,a_{n-1})$ is a realisation of~$p$ if and only if $a_0$ is a solution to~(\ref{de}), $a_\ell=\sigma^\ell(a_0)$ for all $\ell\leq n-1$, and $\trdeg_k(a)=n$.

\begin{claim}
\label{ctod}
If~$(\ref{de})$ satifies~$(C_m)$  then~$p$ is $m$-disintegrated.
\end{claim}
\begin{claimproof}
Let $a_1,\dots,a_m$ be pairwise $\sigma$-disjoint realisations of~$p$, and write each $a_i=(a_{i,0},\dots,a_{i,n-1})$.
Fixing $r\in\ZZ$, note that if $a_{i,0}=\sigma^r(a_{j,0})$ then
$$a_{i,\ell}=\sigma^\ell(a_{i,0})=\sigma^{\ell+r}(a_{j,0})=\sigma^r(a_{j,\ell})$$
for all $\ell=0,\dots,n-1$, so that $a_i=\sigma^r(a_j)$.
Hence $a_{1,0},\dots,a_{m,0}$ are pairwise $\sigma$-disjoint solutions to~(\ref{de}).
Condition $(C_m)$ therefore implies that $k\langle a_{1,0},\dots,a_{m,0}\rangle$ is of transcendence degree $mn$ over~$k$.
But $k\langle a_{1,0},\dots,a_{m,0}\rangle=k(a_1,\dots,a_m)$, proving that $a_1,\dots,a_m$ are independent over~$k$.
(We are using here that for solutions to the rational type~$p$, algebraic independence agrees with independence.)
\end{claimproof}

\begin{claim}
\label{dtoc}
If~$(\ref{de})$ satisfies~$(C_1)$  and~$p$ is $m$-disintegrated then~$(\ref{de})$ satisfies~$(C_m)$.
\end{claim}
\begin{claimproof}
Suppose $b_1,\dots, b_m$ are pairwise $\sigma$-disjoint solutions to~(\ref{de}).
Let $a_i:=(b_i,\sigma(b_i),\dots,\sigma^{n-1}(b_i))$.
Then $a_1,\dots,a_m$ are also pairwise $\sigma$-disjoint.
As $(C_1)$ implies that each $a_i$ is of transcendence degree~$n$ over~$k$, we have that $a_1,\dots, a_m$ are realisations of~$p$.
It follows from $m$-disintegration that $a_1,\dots, a_m$ are independent over~$k$,
so that $k\langle b_1,\dots,b_m\rangle=k(a_1,\dots,a_m)$ is of transcendence degree~$mn$.
\end{claimproof}

Suppose now that~(\ref{de}) satisfies~$(C_{n+4})$.
By Claim~\ref{ctod}, $p$ is $(n+4)$-disintegrated.
We claim that $p$ is totally disintegrated.
Indeed, if $n=1$ then this follows from Theorem~\ref{d3}(a), whereas if $n>1$ then, as Theorem~\ref{thm:nmdeg}(a) tells us that $\nmdeg(p)\leq n+3$, the total disintegration of~$p$ follows from Theorem~\ref{d3}(b).
Claim~\ref{dtoc} then implies that (\ref{de}) satisfies~$(C_m)$ for all~$m$.

Consider, finally, the case when $k\subset\C$.
If~(\ref{de}) satisfies~$(C_3)$ then $p$ is $3$-disintegrated, by Claim~\ref{ctod}, and hence totally disintegrated, by Theorem~\ref{d3}(c).
So, by Claim~\ref{dtoc}, (\ref{de}) satisfies~$(C_m)$ for all~$m$.
\end{proof}

\bigskip
\section{Two examples}
\label{sec:example}

\noindent
Both examples that we wish to describe will make use of the quantifier-free binding groups we studied in~\cite{qfint}, and reviewed in~$\S$\ref{subsect:nonorth} above.

\medskip
\subsection{A nonminimal rational type with exchange}
\label{xexample}

\noindent
We continue to work in characteristic zero.
Fix an algebraically closed subfield of the fixed field, $k\subset\C$, a simple abelian variety~$A$ over~$k$ of dimension $>1$, and a nontorsion element $t\in A(k)$.
Consider the rational dynamical system $(A,\phi)$ where $\phi(x)=x+t$.
Let $p\in S_{\qf}(k)$ be the generic quantifier-free type of $(A,\phi)$.

\begin{proposition}
$p$ satisfies exchange and is not minimal.
\end{proposition}

\begin{proof}
Note, first of all, that $p\perp^w\C$.
Indeed, by~\cite[Proposition~2.7]{qfint}, it suffices to show that every invariant rational function, $f:(A,\phi)\dto(\AA^1,\id)$ over~$k$, is constant.
Fixing $a\in A(\U)$ Zariski generic over~$k$, we have that any such~$f$ is (defined) and constant on $a+\ZZ t$.
But, by simplicity, and as $t$ is nontorsion, $a+\mathbb Zt$ is Zariski dense in~$A$, which forces $f$ to be constant.

Next, observe that $p(\U)$ is a coset of $A(\C)$.
Indeed, fixing $a\models p$ we have that for all $b\models p$, $b-a\in A(\C)$, so that $b\in a+A(\C)$.
Conversely, as $p\perp^w\C$, for any $c\in A(\C)$ we have that $a+c$ is still Zariski generic in~$A$ over~$k$, and
$$\sigma(a+c)=\sigma(a)+c=a+t+c=\phi(a+c)$$
so that $(a+c)\models p$.

In particular, $p$ is not minimal as $\dim A>1$.

Let us show, now, that~$p$ satisfies exchange.
Fix $a\models p$ and $b\in\acl(ka)\setminus k$ such that $\qftp(b/k)$ is rational.
We need to prove that $a\in\acl(kb)$.

Let $L:=\acl(kb)$ and $q:=\qftp(a/L)$.
Note that~$q$ is also qf-internal to the constants as the difference (in~$A$) between any two realisations is contained in $A(\C)$.
Also, $q\perp^w\C$, since $a\ind_k\C$ implies $a\ind_L\C$ since $L\subseteq\acl(ka)$, and using the fact that weak orthogonality is an invariant of the quantifier-free type.
Let $\G:=\aut_{\qf}(q/\C)$.
Consider the homomorphism $\G\to A(\C)$ given by $\alpha\mapsto c_\alpha:=\alpha(a)-a$.
Note that $\alpha(x)=x+c_\alpha$ for all $x\models q$.
Indeed,
$$\alpha(x)=\alpha((x-a)+a))=(x-a)+\alpha(a)=x+c_\alpha$$
where the second equality uses the defining condition of quantifier-free binding groups -- namely (A) of~$\S$\ref{subsect:nonorth} above -- and the fact that $x-a\in A(\C)$.
In particular, $\G\to A(\C)$ is injective.
We now describe the image.

Let $V=\loc(a/L)$.
It is an irreducible invariant subvariety of $(A,\phi)$ over~$L$.
Let~$\I$ denote the set of all irreducible invariant subvarieties of $(V^r,\phi)\times(\AA^\ell,\id)$ over~$L$ that project dominantly onto each copy of $V$, as $r,\ell\geq 0$ range.
For any $r,\ell\geq 0$, let $A$ act on $A^r\times\AA^\ell$ by co-ordinatewise translation on $A^r$ and trivially on $\AA^\ell$.
Consider the set~$B$ of all $g\in A$ that, under this action, preserves every variety in~$\I$.
Then $B$ is an algebraic subgroup of $A$ over~$L$.

We claim that the image of $\G\to A(\C)$ lands in $B(\C)$.
Indeed, suppose $\alpha\in\G$ with image $c_\alpha\in A(\C)$.
Let $Z\subseteq V^r\times\AA^\ell$ be an element of~$\I$.
A generic point of $(Z,\phi\times\id)$ over $L(c_\alpha)$ will be Zariski generic over $L(c_\alpha)$, and will be of the form $(u,c)$ where $u$ is an $r$-tuple of realisation of~$q$ and~$c$ is an $\ell$-tuple from~$\C$.
The defining condition of quanitifer-free binding groups (namely condition~(A) of~$\S$\ref{subsect:nonorth}) and the fact that $\alpha(u)=u+c_\alpha$, ensures that $(u+c_\alpha,c)\in Z$.
Hence $c_\alpha(Z)\subseteq Z$, and so by (absolute) irreducibility of~$Z$ we have $c_\alpha(Z)= Z$.
That is, $c_\alpha\in B(\C)$.

Next, we claim that the image of $\G\to B(\C)$ is surjective.
Suppose $d\in B(\C)$.
We first note that translation by~$d$ is a permutation of $q(\U)$.
Fixing $x\models q$, we have that $d+x\in V$ because being in~$B$ in particular implies that translation by~$d$ preserves~$V$.
On the other hand,
$$\trdeg(d+x/L)\geq \trdeg(d+x/L(d))=\trdeg(x/L(d)=\trdeg(x/L)=\dim V$$
where the penultimate equality uses that $q\perp^w\C$ and~$d$ is a $\C$-point.
Hence $d+x$ realises~$q$.
As a similar argument works with $-d$, we have that translation by $d$ is a permutation of $q(\U)$.
To check it is an element of $\aut_{\qf}(q/\C)$, fix a tuple~\(u\) of realisations of~\(q\), and a tuple~\(c\) of elements of~\(\C\).
Since $(u,c)$ and $(u+d,c)$ are both in $(V^r\times\AA^\ell, \phi\times\id)^\sharp$, to show that they have the same quantifier-free type over~$L$ it suffices to show that they have the same Zariski locus.
But $\loc(u,c/L)\in\I$, and hence is preserved by the action of~$d$.
So $\loc(u+d,c/L)\subseteq \loc(u,c/L)$.
As a similar argument works with $-d$, we have that $\loc(u+d,c/L)=\loc(u,c/L)$, as desired.

So $\G\to B(\C)$ is an isomorphism.
If $B=A$ then the $\G$-orbit of $a$ is $a+A(\C)=p(\U)$, which implies that $p$ isolates~$q$.
In particular, this would mean that $a\ind_kL$, contradicting the fact that $b\in\acl(ka)\setminus k$.
Hence $B$ is a proper algebraic subgroup.
As $A$ is a simple abelian variety, this forces $B$ to be finite, and hence~$\G$ is finite.
But then the full automorphism group, $\aut(q/\C)$, which is contained in $\G=\aut_{\qf}(q/\C)$, is finite.
So $a\in\acl(L\C)$.
As $q\perp^w\C$, this implies $a\in L=\acl(kb)$, as desired.
\end{proof}

\medskip
\subsection{A $2$-transitive binding group action}
\label{sec-2trans}
Our goal here is to produce examples which exhibit two phenomena:
\begin{itemize}
\item
degree of nonminimality greater than~$1$, and 
\item
$2$-disintegration without total disintegration.
\end{itemize}
The key will be a $2$-transitive binding group action.

\begin{proposition}
\label{2trans}
Suppose $(V,\phi)$ is an isotrivial rational $\sigma$-variety with binding group $(G,\rho)$ such that the action of~$G$ on~$V$ is by automorphisms and is $2$-transitive.
Let $p$ be the generic quantifier-free type of $(V,\phi)$.
Then $p$ is $2$-disintegrated but not totally disintegrated.
Moreover, if $\dim V>1$ then $\nmdeg(p)>1$.
\end{proposition}

\begin{proof}
Since~$p$ is qf-internal to the fixed field it is not one-based, and hence is not totally disintegrated -- see Remark~\ref{rem:disint}(d).

Suppose $Z\subseteq V^2$ is a proper irreducible invariant subvariety projecting dominantly onto both copies of $V$.
By the defining condition of the binding group of $(V,\phi)$, so condition~(B) of~$\S$\ref{subsect:nonorth}, each element of~$G$ must preserve~$Z$.
The $2$-transitivity of the action of~$G$ on~$V$ then forces $Z$ to be the diagonal.
This implies that any pair of distinct realisations of~$p$ must be independent.
So~$p$ is $2$-disintegrated.

Next, suppose $\dim V>1$.
That, together with isotriviality, rules out $(V,\phi)$ being simple.
So we can apply Proposition~\ref{prop:nmdeg-geom} which tells us that to show $\nmdeg(p)>1$ we need only show that $(V,\phi)$ is $(V,\phi)$-simple.
(See Definition~\ref{def:wsimple}.)
But that also follows from the above observation about invariant subvarieties of $V^2$.\end{proof}

So we need to find $(V,\phi)$ as in the hypotheses of Proposition~\ref{2trans}.
In fact, we will produce such an example in every dimension.

We work over an algebraically closed inversive difference field \((k,\sigma)\).
Fix $n\geq 2$ and \(A\in\GL_n(k)\), and 
consider the \(\sigma\)-variety \(\pp{\AA^n,\phi_A}\), where 
\(\phi_A(v)=Av\).
(This is the usual setup of linear difference equations, as 
in~\cite{sp}.) 
Observe that \(\pp{\AA^n,\phi_A}\) is isotrivial.
Indeed, letting $(\GL_n,\psi_A)$ be the $\sigma$-variety where $\psi_A(B)=AB$, we obtain a trivialisation as follows
\[
\xymatrix{
(\AA^n\times\GL_n,\phi_A\times\psi_A)\ar[dr]\ar[rr]^{g}_{\isom}&&
(\AA^n\times\GL_n,\id\times \psi_A)\ar[dl]\\
&(\GL_n,\psi_A)
}
\]
with \(g(v,Z):=(Z^{-1}v,Z)\).

Let $(G,\rho)$ be the binding group of $(\AA^n,\phi_A)$.
Note that vector addition is an equivariant map
$(\AA^n,\phi_A)\times(\AA^n,\phi_A)\to(\AA^n,\phi_A)$,
and that scalar multiplication 
is an equivariant map
$(\AA,\id)\times(\AA^n,\phi_A)\to(\AA^n,\phi_A)$.
The elements of the binding group will preserve the graphs of these operations.
Each element of~$G$ thus extends uniquely to a linear transformation of $\AA^n$, and we can, and do, identify $G$ with an algebraic subgroup of 
\(\GL_n\).

The image~\(a\) of~\(A\) in \(\pgl_n(k)\) similarly determines an 
automorphism \(\phi_a\) of the associated projective space 
\(\PP^{n-1}\), which is similarly isotrivial, with binding group~\((\overline{G},\bar \rho)\) for $\overline G\leq\pgl_n$ an algebraic subgroup.

\begin{lemma}
  The projection from \(\GL_n\) to \(\pgl_n\) maps~$G$ to $\overline G$.
  \end{lemma}
\begin{proof}
	Let \(\bar{g}\in\pgl_n\) be the image of some \(g\in{}G\), 
	and assume that \(X\) is an irreducible invariant subvariety of $((\PP^{n-1})^r,\phi_a)\times (W,\id)\) for some $r\geq 1$ and trivial dynamics $(W,\id)$.
	Then its 
	pre-image in \((\AA^n)^r\times W\) is also irreducible and invariant for $\phi_A\times\id$, so is preserved by \(g\).
	Hence \(X\) is preserved by \(\bar{g}\).
	That is, $\bar g\in\overline G$.
\end{proof}

It follows that if $G=\GL_n$ then $\overline G=\pgl_n$.
Since the action of $\pgl_n$ on $\PP^{n-1}$ is $2$-transitive, $(\PP^{n-1},\phi_a)$ will in that case be our desired example of an isotrivial rational $\sigma$-variety with binding group acting $2$-transitively.

To complete the example, then, we need to know that, for some $(k,\sigma)$, there is $A\in\GL_n(k)$ such that $(\AA^n,\phi_A)$ has binding group all of $\GL_n$.
This follows 
from~\cite[Theorem~3.1]{sp}, an instance of which is that, working over \(k=\CC(t)^{\alg}\) with 
\(\sigma(t)=t+1\), there is $A\in\GL_n(k)$ such that the ``difference Galois group" of the equation $\sigma(v)=Av$ is $\GL_n$.
Now, while we have not shown that our binding group coincides in the 
	linear case with the Galois group as defined in~\cite{sp},
this is in fact the case and follows from~\cite{Moshe}.

Finally, let us point out that our example, when $n>2$, has degree of nonminimality~$2$.
To see this, using Proposition~\ref{prop:nmdeg-geom}, we need only show that $(\PP^{n-1},\phi_a)$ is not $((\PP^{n-1})^2,\phi_a)$-simple.
This is done by observing that the set of collinear triples in $\PP^{n-1}$ (so coplanar triples of lines through the origin in $\AA^n$) forms a proper invariant irreducible subvariety of $(\PP^{n-1})^3$ that projects onto each $\PP^{n-1}$-co\"ordinate and has dimension $3n-4$ (which is strictly bigger than $2n-2$ as $n>2$).

\vfill\pagebreak
\bibliographystyle{plain}


\begin{thebibliography}{10}

\bibitem{bfm}
John Baldwin, James Freitag, and Scott Mutchnik.
\newblock Simple homogeneous structures and indiscernible sequence invariants.
\newblock Preprint, arXiv:2405.08211, 2024.

\bibitem{bms}
Jason Bell, Rahim Moosa, and Matthew Satriano.
\newblock On invariant rational functions under rational transformations.
\newblock {\em Selecta Math. (N.S.)}, 30(3):Paper No. 53, 23, 2024.

\bibitem{COP}
Fr\'{e}d\'{e}ric Campana, Keiji Oguiso, and Thomas Peternell.
\newblock Non-algebraic hyperk\"{a}hler manifolds.
\newblock {\em J. Differential Geom.}, 85(3):397--424, 2010.

\bibitem{zoe-psf}
Zo\'{e} Chatzidakis.
\newblock Model theory of finite fields and pseudo-finite fields.
\newblock {\em Ann. Pure Appl. Logic}, 88(2-3):95--108, 1997.

\bibitem{zoe-cbp}
Zo\'{e} Chatzidakis.
\newblock A note on canonical bases and one-based types in supersimple
  theories.
\newblock {\em Confluentes Math.}, 4(3):1250004, 34, 2012.

\bibitem{acfa}
Zo\'{e} Chatzidakis and Ehud Hrushovski.
\newblock Model theory of difference fields.
\newblock {\em Trans. Amer. Math. Soc.}, 351(8):2997--3071, 1999.

\bibitem{ch1}
Zo\'{e} Chatzidakis and Ehud Hrushovski.
\newblock Difference fields and descent in algebraic dynamics. {I}.
\newblock {\em J. Inst. Math. Jussieu}, 7(4):653--686, 2008.

\bibitem{ch2}
Zo\'{e} Chatzidakis and Ehud Hrushovski.
\newblock Difference fields and descent in algebraic dynamics. {II}.
\newblock {\em J. Inst. Math. Jussieu}, 7(4):687--704, 2008.

\bibitem{acfa2}
Zo\'{e} Chatzidakis, Ehud Hrushovski, and Ya'acov Peterzil.
\newblock Model theory of difference fields. {II}. {P}eriodic ideals and the
  trichotomy in all characteristics.
\newblock {\em Proc. London Math. Soc. (3)}, 85(2):257--311, 2002.

\bibitem{c3}
James Freitag, R\'{e}mi Jaoui, and Rahim Moosa.
\newblock When any three solutions are independent.
\newblock {\em Invent. Math.}, 230(3):1249--1265, 2022.

\bibitem{nmdeg2}
James Freitag, R\'{e}mi Jaoui, and Rahim Moosa.
\newblock The degree of nonminimality is at most 2.
\newblock {\em J. Math. Log.}, 23(3):Paper No. 2250031, 6, 2023.

\bibitem{nmdeg}
James Freitag and Rahim Moosa.
\newblock Bounding nonminimality and a conjecture of {B}orovik-{C}herlin.
\newblock {\em J. Eur. Math. Soc. (JEMS)}, 27(2):589--613, 2025.

\bibitem{mm}
Ehud Hrushovski.
\newblock The {M}anin-{M}umford conjecture and the model theory of difference
  fields.
\newblock {\em Ann. Pure Appl. Logic}, 112(1):43--115, 2001.

\bibitem{Moshe}
Moshe Kamensky.
\newblock Definable groups of partial automorphisms.
\newblock {\em Selecta Math. (N.S.)}, 15(2):295--341, 2009.

\bibitem{qfint}
Moshe Kamensky and Rahim Moosa.
\newblock Binding groups for algebraic dynamics.
\newblock Preprint, arXiv:2405.06092.

\bibitem{alicetom}
Alice Medvedev and Thomas Scanlon.
\newblock Invariant varieties for polynomial dynamical systems.
\newblock {\em Ann. of Math. (2)}, 179(1):81--177, 2014.

\bibitem{moosa-pillay2014}
Rahim Moosa and Anand Pillay.
\newblock Some model theory of fibrations and algebraic reductions.
\newblock {\em Selecta Math. (N.S.)}, 20(4):1067--1082, 2014.

\bibitem{pillay-ziegler}
Anand Pillay and Martin Ziegler.
\newblock Jet spaces of varieties over differential and difference fields.
\newblock {\em Selecta Math. (N.S.)}, 9(4):579--599, 2003.

\bibitem{sp}
Marius van~der Put and Michael~F. Singer.
\newblock {\em Galois theory of difference equations}.
\newblock Number 1666 in Lecture Notes in Mathematics. Springer-Verlag, Berlin,
  1997.

\end{thebibliography}

\end{document}